\documentclass[final,3p,times,12pt]{elsarticle}
\usepackage{graphicx,amssymb, verbatim,enumerate,amsmath}
\usepackage{amsthm}
\usepackage[pdfpagelabels,plainpages=false]{hyperref}
\hypersetup{ linkcolor =blue,
 citecolor = cyan,
 urlcolor = blue,
colorlinks = true, bookmarksopen = true, bookmarksnumbered = true,
bookmarksopenlevel = {1}, linktocpage = true,
pdfstartview = FitH,
pdftitle = {Divided-difference equation, bivariate Racah polynomials, bivariate Wilson polynomials, bivariate continuous dual Hahn polynomials, bivariate continuous Hahn polynomials, multivariate polynomials},
pdfsubject
= {article},
pdfauthor = {Area et al.},
pdfcreator = {pdftextify}, pdfproducer = {pdftextify}
}

\newcommand{\hyper}[5]{\,{}_{#1}F_{#2}\left(\!\!%
\begin{array}{cc}{\displaystyle{#3}}\\[-0.1ex]
{\displaystyle{#4}} \end{array}\Big| \,{\displaystyle{#5}}
\right)}

\newtheorem{theorem}{Theorem}
\newtheorem{corollary}[theorem]{Corollary}

\newtheorem{proposition}[theorem]{Proposition}
\newtheorem{conjecture}[theorem]{Conjecture}

\newtheorem{remark}[theorem]{Remark}

\journal{}
\begin{document}

\begin{frontmatter}
\title{Linear partial divided-difference equation satisfied by multivariate orthogonal polynomials on quadratic lattices}

\author[label1]{D. D.~Tcheutia}
\ead{duvtcheutia@yahoo.fr}

\author[label5]{Y.~Guemo Tefo}
\ead{guemoyves2000@yahoo.fr}

\author[label1,label4]{M.~Foupouagnigni}
\ead{foupouagnigni@gmail.com}

\author[label3]{E.~Godoy}
\ead{egodoy@dma.uvigo.es}

\author[label2]{I.~Area\corref{cor1}}
\ead{area@uvigo.es}
\cortext[cor1]{Corresponding author}

\address[label1]{African Institute for Mathematical Sciences, AIMS-Cameroon, P.O. Box 608, Limb\'e Crystal Gardens, South West Region, Cameroon}
\address[label5]{Department of Mathematics, Faculty of Sciences, University of Yaounde I, Yaound\'e, Cameroon}
\address[label4]{Department of Mathematics, Higher Teachers' Training College, University of Yaounde I, Yaound\'e, Cameroon}
\address[label3]{Departamento de Matem\'atica Aplicada II, E.E. Industrial, Universidade de Vigo, Campus Lagoas-Marcosende, 36310 Vigo, Spain}
\address[label2]{Departamento de Matem\'atica Aplicada II, E.E. Telecomunicaci\'on, Universidade de Vigo, Campus Lagoas-Marcosende, 36310 Vigo, Spain}

\begin{abstract}
In this paper, a fourth-order partial divided-difference equation on quadratic lattices with polynomial coefficients satisfied by bivariate Racah polynomials is presented. From this equation we obtain explicitly the matrix coefficients appearing in the three-term recurrence relations satisfied by any bivariate orthogonal polynomial solution of the equation. In particular, we provide explicit expressions for the matrices in the three-term recurrence relations satisfied by the bivariate Racah polynomials introduced by Tratnik. Moreover, we present the family of monic bivariate Racah polynomials defined from the three-term recurrence relations they satisfy, and we solve the connection problem between two different families of bivariate Racah polynomials. These results are then applied to other families of bivariate orthogonal polynomials, namely the bivariate Wilson, continuous dual Hahn and continuous Hahn, the latter two through limiting processes. The fourth-order partial divided-difference equations on quadratic lattices are shown to be of hypergeometric type in the sense that the divided-difference derivatives of solutions are themselves solution of the same type of divided-difference equations.
\end{abstract}


\end{frontmatter}

\section{Introduction}
Univariate Racah polynomials can be defined in terms of hypergeometric series as \cite[page 190]{MR2656096}
\begin{multline}\label{eq:urp}
r_{n}(\alpha,\beta,\gamma,\delta;s)=r_{n}(s)=(\alpha+1)_{n}\,(\beta+\delta+1)_{n}\,(\gamma+1)_{n} \\
\times
\hyper{4}{3}{-n,n+\alpha+\beta+1,-s,s+\gamma+\delta+1}{\alpha+1,\beta+\delta+1,\gamma+1}{1},\quad
n=0,1,\ldots,N,
\end{multline}
where $r_{n}(\alpha,\beta,\gamma,\delta;s)$ is a polynomial of degree $2n$ in $s$ and of degree $n$ in the quadratic lattice \cite{MR1342384,MR1149380}
\begin{equation}\label{eq:defeta}
\eta(s)=s(s+\gamma+\delta+1),
\end{equation}
and $(A)_{n}=A(A+1)\cdots(A+n-1)$ with $(A)_{0}=1$ denotes the Pochhammer symbol. Univariate Racah polynomials satisfy the following second-order linear divided-difference equation
\cite{MR2383000}
\begin{equation}\label{eqdiv-diff}
\phi(\eta(s)) {\mathbb{D}}_{\eta}^{2} r_{n}(s)+ \tau(\eta(s)) {\mathbb{S}}_{\eta} {\mathbb{D}}_{\eta} r_{n}(s)+ \lambda_{n} r_{n}(s)=0,
\end{equation}
where $\phi$ is a polynomial of degree two in the lattice $\eta(s)$ given by
\begin{multline*}
\phi(\eta(s))=-(\eta(s))^{2} + \frac{1}{2} (-\alpha (2 \beta +\delta +\gamma +3)+\beta (\delta -\gamma -3)-2 (\delta \gamma +\delta +\gamma +2)) \eta(s) \\
-\frac{1}{2} (\alpha +1) (\gamma +1) (\beta +\delta +1) (\delta +\gamma +1),
\end{multline*}
$\tau$ is a polynomial of degree one in the lattice $\eta(s)$ given by
\begin{equation*}
\tau(\eta(s))=-(\alpha +\beta +2) \eta(s)-(\alpha +1) (\gamma +1) (\beta+\delta +1),
\end{equation*}
the eigenvalues $\lambda_{n}$ are given by
\begin{equation*}
\lambda_{n}=n (\alpha +\beta +n+1),
\end{equation*}
and the difference operators ${\mathbb{D}}_{\eta}$ and ${\mathbb{S}}_{\eta}$ \cite{MR973434,MR1379135,Wittearxiv} are defined by
\begin{equation}\label{eq:doperator}
{\mathbb{D}}_{\eta} f(s)=\frac{f(s+1/2)-f(s-1/2)}{\eta(s+1/2)-\eta(s-1/2)}, \quad
{\mathbb{S}}_{\eta} f(s)=\frac{f(s+1/2)+f(s-1/2)}{2}.
\end{equation}
Notice that the above operators transform polynomials of degree $n$ in the lattice $\eta(s)$ defined in \eqref{eq:defeta} into polynomials of respectively degree $n-1$ and $n$ in the same variable $\eta(s)$.

Equation \eqref{eqdiv-diff} can be also written in many other forms, e.g. \cite[Eq. (9.2.5)]{MR2656096}
\[
n(n+\alpha+\beta+1) r_{n}(s) = B(s) r_{n}(s+1) - (B(s)+D(s))r_{n}(s) + D(s) r_{n}(s-1),
\]
where $B(s)$ and $D(s)$ are the rational functions given by
\begin{align*}
B(s)&=\frac{(\alpha +s+1) (\gamma +s+1) (\beta +\delta +s+1) (\delta +\gamma +s+1)}{(\delta +\gamma +2 s+1) (\delta +\gamma +2 s+2)} , \\
D(s)&= \frac{s (\delta +s) (-\beta +\gamma +s) (-\alpha +\delta
+\gamma +s)}{(\delta +\gamma +2 s) (\delta +\gamma +2 s+1)}.
\end{align*}

We would like to notice here that
\begin{equation}\label{eq:dur1}
{\mathbb{D}}_{\eta} r_{n}(\alpha,\beta,\gamma,\delta;s) = n(n+\alpha+\beta+1) r_{n-1}(\alpha+1,\beta+1,\gamma+1,\delta;s-1/2)
\end{equation}
There are another families of univariate orthogonal polynomials on quadratic lattices and we refer to \cite{MR2656096,MR1149380} as basic references on this topic.

\medskip
Multivariable Racah polynomials have been introduced by Tratnik in \cite{MR1122519} and deeply analyzed by Geronimo and Iliev in \cite{MR2784425}, where they construct a commutative algebra ${\mathcal{A}}_{x}$ of difference operators in ${\mathbf{R}}^{p}$, depending on $p+3$ parameters, which is diagonalized by the multivariable Racah polynomials considered by Tratnik. In the particular case $p=2$, the bivariate Racah polynomials are defined in terms of univariate Racah polynomials \eqref{eq:urp} as
\begin{multline}\label{eq:brp}
R_{n,m}(s,t;\beta_{0},\beta_{1},\beta_{2},\beta_3,N)
=r_{n}(\beta_{1}-\beta_{0}-1,\beta_{2}-\beta_{1}-1,-t-1,\beta_{1}+t;s) \\
\times
r_{m}(2n+\beta_{2}-\beta_{0}-1,\beta_{3}-\beta_{2}-1,n-N-1,n+\beta_{2}+N;t-n),
\end{multline}
which are polynomials in the lattices $x(s)=s(s+\beta_{1})$ and $y(t)=t(t+\beta_{2})$. These polynomials coincide with the bivariate Racah polynomials of parameters $a_{1}$, $a_{2}$, $a_{3}$, $\gamma$, and $\eta$ introduced by Tratnik \cite[Eq. (2.1)]{MR1122519} after the substitutions
\begin{equation}\label{substrat:Geronimo}
\beta_{0}=a_{1}-\eta-1,\quad \beta_{1}=a_{1}, \quad \beta_{2}=a_{1}+a_{2}, \quad \beta_{3}=a_{1}+a_{2}+a_{3}, \,\text{ and }\,\, N=-\gamma-1.
 \end{equation}

\medskip
As indicated in \cite{MR2784425} it is a bit surprising that the difference operator having bivariate Racah polynomials as
eigenfunctions has a much more complicated structure than the operator for e.g. bivariate big $q$-Jacobi polynomials
\cite{AAGR2013,MR2559345,Lewanowicz20138790}. The equation for bivariate Racah polynomials given in \cite{MR2784425} has $9$ rational coefficients, compared to $6$ polynomial coefficients for the operator corresponding to bivariate big $q$-Jacobi polynomials. A Lie algebraic description of Tratnik's extension of the Racah polynomials has been presented in \cite{Post}. In \cite[Appendix]{MR2784425} a pair of difference equations for the bivariate Racah polynomials are explicitly given, involving rational coefficients.

\medskip
In Section \ref{sec:2} we rewrite a difference equation involving rational coefficients given in \cite[Appendix]{MR2784425} into a fourth-order linear partial divided-difference equation on quadratric lattices for the bivariate Racah polynomials involving polynomial coefficients in quadratic lattices. This equation being of hypergeometric type, allows us to obtain some properties of the difference derivatives of bivariate Racah polynomials. From this new form of the linear partial divided-difference equation the matrix coefficients appearing in the three-term recurrence relations satisfied by any bivariate orthogonal polynomial solution of this partial divided-difference equation are explicitly given in Section \ref{sec:3}. As a particular case, we provide explicit expressions for the matrix coefficients of the three-term recurrence relations satisfied by the bivariate Racah polynomials introduced by Tratnik \cite{MR1122519} and deeply analyzed by Geronimo and Iliev \cite{MR2784425}. Moreover, we present the family of monic bivariate Racah polynomials defined from the three-term recurrence relations they satisfy. By using this family of monic bivariate Racah polynomials and the analysis for the three-term recurrence relations we explicitly solve the connection problem between the two families of bivariate Racah polynomials introduced by Tratnik \cite{MR1122519}. In Section \ref{sec:4}, by considering appropriate limit relations, we obtain a fourth-order linear partial divided-difference equation for the bivariate Wilson polynomials involving also polynomial coefficients as well as some properties of the difference derivatives of the bivariate Wilson polynomials. Limit relations once more allows us to deduce a fourth-order linear partial divided-difference equations satisfied by the bivariate continuous dual Hahn and continuous Hahn polynomials. Also, properties on the partial difference derivatives of the latter families are also obtained, as well as the matrix coefficients of the three-term recurrence relations they satisfy. Furthermore, we derived a sixth-order linear partial divided-difference equation for the trivariate continuous Hahn polynomials as illustration of the conjecture we state to generalize this result to any multivariate Racah, Wilson and continuous Hahn polynomials.

\section{Fourth-order linear partial divided-difference equation for the bivariate Racah polynomials with polynomial coefficients}\label{sec:2}

In \cite{MR2784425} (see Eq. (3.25) and Appendix for ${\mathcal{L}}_{2}^{x}$ and $\mu_{2}(n)$) the following equation involving rational coefficients for the bivariate Racah polynomials defined in \eqref{eq:brp} up to a normalizing constant is explicitly given
\begin{multline}
\frac{(N-t) \left(\beta_1+s\right) \left(-\beta_0+\beta_1+s\right) \left(\beta_3+N+t\right) \left(\beta_2+s+t\right) \left(\beta_2+s+t+1\right)}{\left(\beta_1+2 s\right) \left(\beta_1+2 s+1\right) \left(\beta_2+2 t\right) \left(\beta_2+2 t+1\right)} (R_{n,m}(s+1,t+1)-{\mathcal{I}}) \\
+\frac{\left(\beta_1+s\right) \left(-\beta_0+\beta_1+s\right) (t-s) \left(\beta_2+s+t\right)}{\left(\beta_1+2 s\right) \left(\beta_1+2
s+1\right) \left(\beta_2+2 t-1\right) \left(\beta_2+2 t+1\right)} \\
\times \left(\left(\beta_2+1\right) \left(\beta_3-1\right)+2 N \left(\beta_3+N\right)+2 t \left(\beta_2+t\right)\right) (R_{n,m}(s+1,t)-{\mathcal{I}}) \\
+ \frac{(N-t) \left(\beta_3+N+t\right) \left(\beta_2+s+t\right) \left(-\beta_1+\beta_2-s+t\right)}{\left(\beta_1+2 s-1\right) \left(\beta_1+2 s+1\right) \left(\beta_2+2 t\right) \left(\beta_2+2 t+1\right)} \\
\times \left(\left(\beta_0+1\right) \left(\beta_1-1\right)+2 s \left(\beta_1+s\right)\right) (R_{n,m}(s,t+1)-{\mathcal{I}}) \\
-\frac{s (N-t) \left(\beta_0+s\right) \left(\beta_3+N+t\right) \left(\beta_1-\beta_2+s-t-1\right)}{\left(\beta_1+2 s-1\right) \left(\beta_1+2 s\right) \left(\beta_2+2 t\right) \left(\beta_2+2 t+1\right)} \left(\beta_1-\beta_2+s-t\right) ({\mathcal{I}}-R_{n,m}(s-1,t+1))\\
+ \frac{\left(\beta_1+s\right) \left(\beta_1-\beta_0+s\right) (s-t) (s-t+1) \left(\beta_2+N+t\right)}{\left(\beta_1+2 s\right) \left(\beta_1+2 s+1\right) \left(\beta_2+2 t-1\right) \left(\beta_2+2 t\right)} \left(\beta_2-\beta_3-N+t\right) ({\mathcal{I}}-R_{n,m}(s+1,t-1))\\
+ \frac{s \left(\beta_0+s\right) \left(\beta_2+N+t\right) \left(\beta_2-\beta_3-N+t\right) \left(\beta_1+s+t-1\right) \left(\beta_1+s+t\right)
}{\left(\beta_1+2 s-1\right) \left(\beta_1+2 s\right) \left(\beta_2+2 t-1\right) \left(\beta_2+2 t\right)} ({\mathcal{I}}-R_{n,m}(s-1,t-1)) \\
+ \frac{s \left(\beta_0+s\right) \left(\left(\beta_2+1\right) \left(\beta_3-1\right)+2 N^2+2 \beta_3 N+2 t^2+2 \beta_2 t\right) }{\left(\beta_1+2 s-1\right) \left(\beta_1+2 s\right) \left(\beta_2+2 t-1\right) \left(\beta_2+2 t+1\right)} \\Ê\times \left(\beta_1+s+t\right) \left(\beta_1-\beta_2+s-t\right) ({\mathcal{I}}-R_{n,m}(s-1,t))\\
-\frac{\left(\left(\beta_0+1\right) \left(\beta_1-1\right)+2 s^2+2 \beta_1 s\right) (s-t) \left(\beta_2+N+t\right) }{\left(\beta_1+2 s-1\right) \left(\beta_1+2 s+1\right) \left(\beta_2+2 t-1\right) \left(\beta_2+2 t\right)} \\ \times \left(\beta_2-\beta_3-N+t\right) \left(\beta_1+s+t\right) ({\mathcal{I}}-R_{n,m}(s,t-1)) \\
+(m+n) (\beta_{3}-\beta_{0}+m+n-1) R_{n,m}(s,t)=0,\label{racah_Geronimo_Iliev}
\end{multline}
where $R_{n,m}(s,t) := R_{n,m}(s,t;\beta_{0},\beta_{1},\beta_{2},\beta_{3},N)$ defined in \eqref{eq:brp} and ${\mathcal{I}}$ denotes the identity operator.

\medskip
 Our objective is to re-write the difference equation for bivariate Racah polynomials by using uniquely polynomial coefficients in the lattices $x(s)$ and $y(t)$. In doing so, we shall consider the lattices $x(s)$ and $y(t)$ are as
\begin{equation}\label{eq:lattices}
x(s)=s(s+\beta_{1}), \quad y(t)=t(t+\beta_{2}).
\end{equation}

\begin{theorem}\label{eq:theorem1}
The bivariate Racah polynomials defined in \eqref{eq:brp} are solution of the following fourth-order linear partial divided-difference equation
\begin{multline}\label{eqdiv-diff-biv}
f_1(x(s),y(t)) \mathbb{D}^2_x\mathbb{D}^2_yR_{n,m}(s,t) + f_2(x(s),y(t))\mathbb{S}_x\mathbb{D}_x\mathbb{D}^2_yR_{n,m}(s,t)
+f_3(x(s),y(t))\mathbb{S}_y\mathbb{D}_y\mathbb{D}^2_x R_{n,m}(s,t)\\
+f_4(x(s),y(t)) \mathbb{S}_x\mathbb{D}_x\mathbb{S}_y\mathbb{D}_y R_{n,m}(s,t)
+ f_5(x(s))\mathbb{D}^2_x R_{n,m}(s,t)+f_6(y(t))\mathbb{D}^2_y R_{n,m}(s,t)\\
+f_7(x(s)) \mathbb{S}_x\mathbb{D}_x R_{n,m}(s,t)
+f_8(y(t))\mathbb{S}_y\mathbb{D}_y R_{n,m}(s,t)+(m+n)(\beta_{3}-\beta_{0}+m+n-1) R_{n,m}(s,t)=0,
\end{multline}
where $R_{n,m}(s,t):=R_{n,m}(s,t;\beta_{0},\beta_{1},\beta_{2},\beta_3,N)$, and the coefficients $f_i$, $i=1,\ldots,8$ are polynomials in the lattices $x(s)$ and $y(t)$ defined in \eqref{eq:lattices} given by
\begin{align*}
f_8(y(t))&=(\beta_0-\beta_3)y(t)-N(\beta_0-\beta_2)(\beta_3+N),\\
f_7(x(s))&=(\beta_0-\beta_3)x(s)-N(\beta_0-\beta_1)(\beta_3+N), \\
f_6(y(t))&=-(y(t))^2+\frac{1}{2}(2N^2+2\beta_3(\beta_0+N)-\beta_2(\beta_3+\beta_0))y(t)-\frac{1}{2}N\beta_2(\beta_0-\beta_2)(\beta_3+N),\\
f_5(x(s))&=-(x(s))^2+\frac{1}{2}(2\beta_3(N+\beta_0)+2N^2-\beta_1(\beta_3+\beta_0))x(s)-\frac{1}{2}N\beta_1(\beta_0-\beta_1)(\beta_3+N), \\
f_4(x(s),y(t))&=-2x(s)y(t)+(2N^2+\beta_2(1-\beta_0)+\beta_3(\beta_0-1+2N)) x(s)\\
& +(\beta_0-\beta_1)(\beta_3+1) y(t)-N(\beta_0-\beta_1)(\beta_2+1)(\beta_3+N), \\
f_3(x(s),y(t))&=(\beta_2-\beta_3)(x(s))^2+x(s)\Big( -(1+\beta_1+\beta_3-2\beta_0)y(t)+ \left( 1+\beta_{{1}}-2\,\beta_{{0}}+\beta_{{2}} \right) {N}^{2}\\
& -\beta_{{3}} \left( -\beta_{{1}}-\beta_{{2}}-1+2\,\beta_{{0}} \right) N+\frac{1}{2}\, \left(\beta_{{2}}-\beta_{{3}} \right) \left( \beta_{{1}}\beta_{{0}}-2\,\beta_{{0}}+\beta_{{1}} \right) \Big)\\
&+\frac{1}{2}\,\beta_{{1}} \left( \beta_{{3}}+1 \right) \left( \beta_{{0}}-\beta_{{1}} \right)y(t)-\frac{1}{2}\,\beta_{{1}}N \left( \beta_{{2}}+1 \right) \left( \beta_{{3}}+N \right) \left( \beta_{{0}}-\beta_{{1}} \right),  \\
f_2(x(s),y(t))&=(\beta_0-\beta_1)(y(t))^2+x(s)\Big((\beta_0+\beta_2-2\beta_3-1)y(t)+\left( 1-\beta_{{0}}+\beta_{{2}} \right) {N}^{2} \\Ê&-\beta_{{3}} \left( -1+\beta_{{0}}-\beta_{{2}} \right) N+\frac{1}{2}\,\beta_{{2}} \left( \beta_{{2}}-\beta_{{3}} \right) \left( \beta_{{0}}-1 \right)\Big) \\
&-\frac{1}{2}\, \left( \beta_{{0}}-\beta_{{1}} \right) \left( 2\,\beta_{{3}}N-\beta_{{3}}\beta_{{2}}+2\,{N}^{2}-2\,\beta_{{3}}+\beta_{{2}} \right)y(t)\\
&-\frac{1}{2}\, \left( \beta_{{0}}-\beta_{{1}} \right) N\beta_{{2}} \left(\beta_{{2}}+1 \right) \left( \beta_{{3}}+N \right),  \\
f_1(x(s),y(t))&=-(x(s))^2y(t)-x(s)(y(t))^2+({N}^{2}+\beta_{{3}}N-\frac{1}{2}\,\beta_{{2}} \left( \beta_{{2}}-\beta_{{3}}\right)) (x(s))^2 \\
&+\frac{1}{2}\,\beta_{{1}} \left( \beta_{{0}}-\beta_{{1}} \right) (y(t))^2 + \Big(\left( \frac{1}{2}\,\beta_{{1}}+\frac{1}{2}-\frac{1}{2}\,\beta_{{3}}-\beta_{{0}} \right) \beta_{{2}}-\beta_{{3}}-\frac{1}{2}\,\beta_{{1}}+2\,\beta_{{0}}\beta_{{3}} \\
&+\beta_{{0}}+{N}^{2}-\beta_{{1}}\beta_{{3}}+\beta_{{3}}N-\frac{1}{2}\,\beta_{{1}}\beta_{{0}}\Big)x(s)y(t) \\
&+\Big(\left( \frac{1}{2}\,\beta_{{1}}\beta_{{0}}+\frac{1}{2}\,{\beta_{{2}}}^{2}+\frac{1}{2}\,\beta_{{1}}\beta_{{2}}+\frac{1}{2}\,\beta_{{2}}-\beta_{{0}}\beta_{{2}}+\frac{1}{2}\,\beta_{{1}}-\beta_{{0}} \right) {N}^{2}
\end{align*}
\begin{align*}
&+\frac{1}{2}\,\beta_{{3}} \left( \beta_{{1}}\beta_{{0}}+{\beta_{{2}}}^{2}+\beta_{{1}}\beta_{{2}}+\beta_{{2}}-2\,\beta_{{0}}\beta_{{2}}+\beta_{{1}}-2\,\beta_{{0}} \right) N  \\
&-\frac{1}{4}\,\beta_{{2}} \left( \beta_{{2}}-\beta_{{3}} \right) \left( \beta_{{1}}\beta_{{0}}+\beta_{{1}}-2\,\beta_{{0}} \right))x(s)\\
&-\frac{1}{4}\,\beta_{{1}} \left( \beta_{{2}}-2\,\beta_{{3}}+2\,\beta_{{3}}N+2\,{N}^{2}-\beta_{{2}}\beta_{{3}} \right) \left( \beta_{{0}}-\beta_{{1}} \right) y(t) \\
&-\frac{1}{4}\,N\beta_{{1}}\beta_{{2}} \left( \beta_{{2}}+1 \right) \left( \beta_{{0}}-\beta_{{1}} \right) \left( \beta_{{3}}+N \right).
\end{align*}
\end{theorem}
\begin{proof}
The result follows from Equation \eqref{racah_Geronimo_Iliev} given in
\cite[Appendix]{MR2784425}, first by writing the nine expressions
\begin{gather*}
 \mathbb{D}^2_x\mathbb{D}^2_yR_{n,m}(s,t) , \quad \mathbb{S}_x\mathbb{D}_x\mathbb{D}^2_yR_{n,m}(s,t), \quad \mathbb{S}_y\mathbb{D}_y\mathbb{D}^2_x R_{n,m}(s,t), \quad \mathbb{S}_x\mathbb{D}_x\mathbb{S}_y\mathbb{D}_y R_{n,m}(s,t), \\
 \mathbb{D}^2_x R_{n,m}(s,t), \quad \mathbb{D}^2_y R_{n,m}(s,t), \quad \quad \mathbb{S}_x\mathbb{D}_x R_{n,m}(s,t),
\quad \mathbb{S}_y\mathbb{D}_y R_{n,m}(s,t), \quad R_{n,m}(s,t),
\end{gather*}
as linear combinations of the nine terms $R_{n,m}(s+\ell_{1},t+\ell_{2})$, $-1 \leq \ell_{1},\ell_{2} \leq 1$, then solving the resulting system of linear equations in terms of the unknowns $R_{n,m}(s+\ell_{1},t+\ell_{2})$, $-1 \leq \ell_{1},\ell_{2} \leq 1$, and finally replacing the solution into equation \eqref{racah_Geronimo_Iliev} and observe that the coefficients obtained are all polynomials in the lattices $x(s)$ and $y(t)$.

\end{proof}

\medskip
In order to give properties of the difference derivatives of the bivariate Racah polynomials we shall need to recall a number of properties of the difference operators $\mathbb{D}_x$ and $\mathbb{S}_x$ (see e.g. \cite{MR2383000} and \cite{FKKM}):
\begin{equation} \label{eq:dxsx}
\begin{cases}
\mathbb{D}_x\mathbb{S}_x f =  \mathbb{S}_x\mathbb{D}_x f+\frac{1}{2}\mathbb{D}^2_x f,\qquad \qquad \,\,
\mathbb{S}^2_x f= \frac{1}{2}\mathbb{S}_x\mathbb{D}_x f+ U_2(s)\mathbb{D}^2_x f+f, \\
\mathbb{D}_x(fg) = \mathbb{S}_x f \mathbb{D}_xg +\mathbb{D}_xf\mathbb{S}_x g, \qquad \,\,\,\,
\mathbb{D}_y\mathbb{S}_y f = \mathbb{S}_y\mathbb{D}_y f+\frac{1}{2}\mathbb{D}^2_y f, \\
\mathbb{S}^2_y f= \frac{1}{2}\mathbb{S}_y\mathbb{D}_y f+ V_2(t)\mathbb{D}^2_y f+f, \qquad
\mathbb{D}_y(fg) = \mathbb{S}_y f \mathbb{D}_y g +\mathbb{D}_y f\mathbb{S}_y g,
\end{cases}
\end{equation}
with $f:=f(s,t)$, $g:=g(s,t)$, where
\begin{equation}\label{eq:uuvv}
U_{2}(s)=x(s)+\frac{\beta_1^2}{4},\quad
V_{2}(t)=y(t)+\frac{\beta_2^2}{4}.
\end{equation}
Using the above properties, we show that the polynomials $\mathbb{D}_x R_{n,m}(s,t)$ and $\mathbb{D}_y R_{n,m}(s,t)$ are also solution of a linear partial divided-difference equation of the same type as \eqref{eqdiv-diff-biv}.
\begin{theorem}\label{theorem:2}
The polynomial $R_{n,m}^{(1,0)}(s,t):=\mathbb{D}_x R_{n,m}(s,t)$ is solution of the following fourth-order linear partial divided-difference equation
\begin{multline*}
f_{11}(x(s),y(t)) \mathbb{D}^2_x\mathbb{D}^2_y R_{n,m}^{(1,0)}(s,t)
+ f_{21}(x(s),y(t))\mathbb{S}_x\mathbb{D}_x\mathbb{D}^2_y R_{n,m}^{(1,0)}(s,t)
+f_{31}(x(s),y(t))\mathbb{S}_y\mathbb{D}_y\mathbb{D}^2_x R_{n,m}^{(1,0)}(s,t)\\
+f_{41}(x(s),y(t))\mathbb{S}_x\mathbb{D}_x\mathbb{S}_y\mathbb{D}_y R_{n,m}^{(1,0)}(s,t)
+ f_{51}(x(s))\mathbb{D}^2_x R_{n,m}^{(1,0)}(s,t)
+f_{61}(y(t))\mathbb{D}^2_y R_{n,m}^{(1,0)}(s,t)\\
+f_{71}(x(s)) \mathbb{S}_x\mathbb{D}_x R_{n,m}^{(1,0)}(s,t)
+f_{81}(y(t))\mathbb{S}_y\mathbb{D}_y R_{n,m}^{(1,0)}(s,t)+(m+n-1)
(\beta_{3}-\beta_{0}+m+n) R_{n,m}^{(1,0)}(s,t)=0,
\end{multline*}
where the coefficients $f_{i1},\ i=1,\ldots,8$ are polynomials in the lattices $x(s)$ and $y(t)$ given by
\begin{align*}
f_{81}(y(t))&=f_8(y(t))+\mathbb{D}_x(f_4(x(s),y(t))), \\
\end{align*}
\begin{align*}
f_{71}(x(s))&=\mathbb{S}_x(f_7(x(s)))+\frac{1}{2}\mathbb{D}_x(f_7(x(s)))+\mathbb{D}_x(f_5(x(s))),\\
f_{61}(y(t))&=f_6(y(t))+\mathbb{D}_x(f_2(x(s),y(t))),\\
f_{51}(x(s))&=\mathbb{S}_x(f_5(x(s)))+ \mathbb{D}_x(f_7(x(s)))U_{2}(s)+\frac{1}{2}\mathbb{S}_x(f_7(x(s))), \\
f_{41}(x(s),y(t))&=\frac{1}{2}\mathbb{D}_x(f_4(x(s),y(t)))+\mathbb{D}_x(f_3(x(s),y(t)))+ \mathbb{S}_x(f_4(x(s),y(t)))\\
f_{31}(x(s),y(t))&=\frac{1}{2}\mathbb{S}_x(f_4(x(s),y(t)))+\mathbb{S}_x(f_3(x(s),y(t)))+ \mathbb{D}_x(f_4(x(s),y(t)))U_{2}(s), \\
f_{21}(x(s),y(t))&=\frac{1}{2}\mathbb{D}_x(f_2(x(s),y(t)))+\mathbb{D}_x(f_1(x(s),y(t)))+ \mathbb{S}_x(f_2(x(s),y(t))),\\
f_{11}(x(s),y(t))&=\frac{1}{2}\mathbb{S}_x(f_2(x(s),y(t)))+\mathbb{S}_x(f_1(x(s),y(t)))+\mathbb{D}_x(f_2(x(s),y(t)))U_{2}(s),
\end{align*}
where $U_{2}(s)$ is given in \eqref{eq:uuvv},
and the polynomial $R_{n,m}^{(0,1)}(s,t):=\mathbb{D}_yR_{n,m}(x(s),y(t))$ is solution of the following fourth-order linear partial divided-difference equation
\begin{multline*}
f_{12}(x(s),y(t)) \mathbb{D}^2_x\mathbb{D}^2_y R_{n,m}^{(0,1)}(s,t)
+ f_{22}(x(s),y(t))\mathbb{S}_x\mathbb{D}_x\mathbb{D}^2_y R_{n,m}^{(0,1)}(s,t)
+f_{32}(x(s),y(t))\mathbb{S}_y\mathbb{D}_y\mathbb{D}^2_x R_{n,m}^{(0,1)}(s,t)\\
+f_{42}(x(s),y(t)) \mathbb{S}_x\mathbb{D}_x\mathbb{S}_y\mathbb{D}_y R_{n,m}^{(0,1)}(s,t)
+ f_{52}(x(s))\mathbb{D}^2_x R_{n,m}^{(0,1)}(s,t)
+f_{62}(y(t))\mathbb{D}^2_y R_{n,m}^{(0,1)}(s,t) \\
+f_{72}(x(s)) \mathbb{S}_x\mathbb{D}_x R_{n,m}^{(0,1)}(s,t)
+f_{82}(y(t))\mathbb{S}_y\mathbb{D}_y R_{n,m}^{(0,1)}(s,t)
 +(m+n-1)(\beta_{3}-\beta_{0}+m+n) R_{n,m}^{(0,1)}(s,t)=0,
\end{multline*}
and the coefficients $f_{i2},\ i=1,\ldots,8$ are polynomials in the lattices $x(s)$ and $y(t)$ given by
\begin{align*}
f_{82}(y(t))&= \mathbb{S}_y(f_8(y(t)))+\frac{1}{2}\mathbb{D}_y(f_8(y(t)))+\mathbb{D}_y(f_6(y(t))),\\
f_{72}(x(s))&=f_7(x(s))+\mathbb{D}_y(f_4(x(s),y(t))),\\
f_{62}(y(t))&= \mathbb{D}_y(f_8(y(t))) V_{2}(t)+\frac{1}{2}\mathbb{S}_y(f_8(y(t)))+\mathbb{S}_y(f_6(y(t))),\\
f_{52}(x(s))&=f_5(x(s))+\mathbb{D}_y(f_3(x(s),y(t))), \\
f_{42}(x(s),y(t))&=\frac{1}{2}\mathbb{D}_y(f_4(x(s),y(t)))+\mathbb{D}_y(f_2(x(s),y(t)))+ \mathbb{S}_y(f_4(x(s),y(t))), \\
f_{32}(x(s),y(t))&=\frac{1}{2}\mathbb{D}_y(f_3(x(s),y(t)))+\mathbb{D}_y(f_1(x(s),y(t)))+ \mathbb{S}_y(f_3(x(s),y(t))), \\
f_{22}(x(s),y(t))&=\frac{1}{2}\mathbb{S}_y(f_4(x(s),y(t)))+\mathbb{S}_y(f_2(x(s),y(t)))+ \mathbb{D}_y(f_4(x(s),y(t)))V_{2}(t),\\
f_{12}(x(s),y(t))&=\frac{1}{2}\mathbb{S}_y(f_3(x(s),y(t)))+\mathbb{S}_y(f_1(x(s),y(t)))+\mathbb{D}_y(f_3(x(s),y(t)))V_{2}(t),
\end{align*}
where $V_{2}(t)$ is given in \eqref{eq:uuvv}.
\end{theorem}
\begin{proof}
The result is obtained by applying the operators $\mathbb{D}_x$ and $\mathbb{D}_y$ to the partial divided-difference equation \eqref{eqdiv-diff-biv} and using the relations \eqref{eq:dxsx}.
\end{proof}

Notice that for $i=1,\dots,8$ we have
\begin{align}
f_{i1}(x(s),y(t))&=f_{i}(x(s-1/2),y(t-1);\beta_{0},1+\beta_{1},2+\beta_{2},2+\beta_{3},N-1),\\
f_{i2}(x(s),y(t))&=f_{i}(x(s),y(t-1/2);\beta_{0},\beta_{1},\beta_{2}+1,\beta_{3}+2,N-1),\label{eq:fi2}
\end{align}
where $f_i=f_i(x(s),y(t);\beta_0,\beta_1,\beta_2,\beta_3,N)$ are already explicitly given in Theorem \ref{eq:theorem1}.

As a consequence of the latter equalities, we obtain the following relation which is the bivariate analogue of \eqref{eq:dur1}
\begin{multline*}
\mathbb{D}_x R_{n,m}(s,t;\beta_{0},\beta_{1},\beta_{2},\beta_3,N)= n(n-\beta_{0}+\beta_{2}-1) \\Ê\times
R_{n-1,m}(s-1/2,t-1;\beta_{0},\beta_{1}+1,\beta_{2}+2,\beta_3+2,N-1).
\end{multline*}

\begin{remark}
If we consider the following second family of the bivariate Racah polynomials obtained from \cite[Equation (2.12)]{MR1122519} using the transformations \eqref{substrat:Geronimo}
\begin{multline}\label{eq:secondfamilyracah}
\bar{R}_{n,m}(s,t;\beta_{0},\beta_{1},\beta_{2},\beta_3,N)=r_{n}(
{2m-\beta_1+\beta_3-1},
{\beta_1-\beta_0-1},
{m-N-1},
{m-N-\beta_{1}},
{N-m-s}) \\
\times r_{m}(\beta_3-\beta_2-1,
\beta_2-\beta_1-1,
{s-N-1},
{-\beta_2-N-s},
N-t),
\end{multline}
it follows that
\begin{equation*}
\mathbb{D}_y \bar{R}_{n,m}(s,t;\beta_{0},\beta_{1},\beta_{2},\beta_3,N) = m(m+\beta_3-\beta_1-1)
\bar{R}_{n,m-1}(s,t-1/2;\beta_{0},\beta_{1},\beta_{2}+1,\beta_3+2,N-1).
\end{equation*}
We should also note that both families of bivariate Racah polynomials $R_{n,m}(s,t;\beta_{0},\beta_{1},\beta_{2},\beta_3,N)$ and $\bar{R}_{n,m}(s,t;\beta_{0},\beta_{1},\beta_{2},\beta_3,N)$ are solution of the same divided-difference equation \eqref{eqdiv-diff-biv}.
\end{remark}
In \cite{MR2784425} (see Eq. (3.25) and Appendix for ${\mathcal{L}}_{1}^{x}$ and $\mu_{1}(n)$), it is shown that the bivariate Racah polynomials $R_{n,m}(s,t):= R_{n,m}(s,t;\beta_{0},\beta_{1},\beta_{2},\beta_3,N)$ are also solution of the difference equation
\begin{multline*}
\frac { \left( s+\beta_{{1}}-\beta_{{0}} \right) \left( s+\beta_{{1}} \right) \left( t+s+\beta_{{2}} \right) \left( t-s \right) }{
 \left( 2\,s+\beta_{{1}} \right) \left( 2\,s+\beta_{{1}}+1 \right) }(R_{n,m}(s+1,t)-R_{n,m}(s,t)) \\
 +{\frac { \left( s+\beta_{{0}} \right) s \left( t-s-\beta_{{1}}+\beta_{{2}} \right) \left( t+s+\beta_{{1}} \right) }{ \left( 2\,s+\beta_{{1}} \right)\left( 2\,s+\beta_{{1}}+1 \right) }}(R_{n,m}(s-1,t)-R_{n,m}(s,t))+n(\beta_2-\beta_0+n-1)R_{n,m}(s,t)=0.
\end{multline*}
Proceeding as above, we can rewrite this equation in the lattices $x(s)$ and $y(t)$ as
\begin{multline*}
(-(x(s))^2+x(s)y(t)+\Big(\beta_0\beta_2-\frac{\beta_1}{2}(\beta_2+\beta_0)\Big)x(s)+\frac{\beta_1}{2}(\beta_1-\beta_0)y(t))\mathbb{D}_x^2R_{n,m}(s,t)\\
+((\beta_0-\beta_2)x(s)+(\beta_1-\beta_0)y(t))\mathbb{S}_x\mathbb{D}_xR_{n,m}(s,t)+n(\beta_2-\beta_0+n-1)R_{n,m}(s,t)=0.
\end{multline*}

\begin{corollary}
The polynomial $R_{n,m}^{(1,1)}(s,t):=\mathbb{D}_x\mathbb{D}_yR_{n,m}(s,t)$ is solution of the fourth-order linear partial divided-difference equation
\begin{multline*}
f_{13}(x(s),y(t)) \mathbb{D}^2_x\mathbb{D}^2_y R_{n,m}^{(1,1)}(s,t)
+f_{23}(x(s),y(t))\mathbb{S}_x\mathbb{D}_x\mathbb{D}^2_y R_{n,m}^{(1,1)}(s,t)
+f_{33}(x(s),y(t))\mathbb{S}_y\mathbb{D}_y\mathbb{D}^2_x R_{n,m}^{(1,1)}(s,t)\\
+f_{43}(x(s),y(t)) \mathbb{S}_x\mathbb{D}_x\mathbb{S}_y\mathbb{D}_y R_{n,m}^{(1,1)}(s,t)
+ f_{53}(x(s))\mathbb{D}^2_x R_{n,m}^{(1,1)}(s,t)
+f_{63}(y(t))\mathbb{D}^2_y R_{n,m}^{(1,1)}(s,t)\\
+f_{73}(x(s)) \mathbb{S}_x\mathbb{D}_x R_{n,m}^{(1,1)}(s,t)
+f_{83}(y(t))\mathbb{S}_y\mathbb{D}_y R_{n,m}^{(1,1)}(s,t)+(m+n-2)
(\beta_{3}-\beta_{0}+m+n+1) R_{n,m}^{(1,1)}(s,t)=0,
\end{multline*}
where the coefficients are polynomials in the lattices $x(s)$ and $y(t)$ given by
\begin{align*}
f_{83}(y(t))&=f_{82}(y(t))+\mathbb{D}_x(f_{42}(x(s),y(t))),\\
f_{73}(x(s))&=\mathbb{S}_x(f_{72}(x(s)))+\frac{1}{2}\mathbb{D}_x(f_{72}(x(s)))+\mathbb{D}_x(f_{52}(x(s))),
\end{align*}
\begin{align*}
f_{63}(y(t))&=f_{62}(y(t))+\mathbb{D}_x(f_{22}(x(s),y(t))), \\
f_{53}(x(s))&=\mathbb{S}_x(f_{52}(x(s)))+ \mathbb{D}_x(f_{72}(x(s)))U_{2}(s)+\frac{1}{2}\mathbb{S}_x(f_{72}(x(s))),  \\
f_{43}(x(s),y(t))&=\frac{1}{2}\mathbb{D}_x(f_{42}(x(s),y(t)))+\mathbb{D}_x(f_{32}(x(s),y(t)))+ \mathbb{S}_x(f_{42}(x(s),y(t))), \\
f_{33}(x(s),y(t))&=\frac{1}{2}\mathbb{S}_x(f_{42}(x(s),y(t)))+\mathbb{S}_x(f_{32}(x(s),y(t)))+ \mathbb{D}_x(f_{42}(x(s),y(t)))U_{2}(s),  \\
f_{23}(x(s),y(t))&=\frac{1}{2}\mathbb{D}_x(f_{22}(x(s),y(t)))+\mathbb{D}_x(f_{12}(x(s),y(t)))+ \mathbb{S}_x(f_{22}(x(s),y(t))), \\
f_{13}(x(s),y(t))&=\frac{1}{2}\mathbb{S}_x(f_{22}(x(s),y(t)))+\mathbb{S}_x(f_{12}(x(s),y(t)))+\mathbb{D}_x(f_{22}(x(s),y(t)))U_{2}(s),
\end{align*}
where $U_{2}(s)$ is given in \eqref{eq:uuvv}.
\end{corollary}
\begin{proof}
The result follows from Theorem \ref{theorem:2}.
\end{proof}

\begin{remark}
\begin{enumerate}
\item An interesting consequence of the above results is that the fourth-order linear partial divided-difference equation \eqref{eqdiv-diff-biv} is of hypergeometric type, i.e. the difference derivatives of a solution are solution of an equation of the same type.
\item We would like to notice that the eigenvalues of the equations satisfied by the polynomials $R^{(1,0)}_{n,m}(s,t)$, $R^{(0,1)}_{n,m}(s,t)$, and $ R^{(1,1)}_{n,m}(s,t)$ are given respectively by $\lambda_{m,n}+\mathbb{D}_xf_7(x(s))$, $\lambda_{m,n}+\mathbb{D}_yf_8(y(t))$, and
\begin{equation*}
\lambda_{m,n}+\mathbb{D}_yf_8(y(t))+\mathbb{D}_xf_{72}(x(s)) =
\lambda_{m,n}+\mathbb{D}_yf_8(y(s))+\mathbb{D}_xf_7(x(s))+\mathbb{D}_x\mathbb{D}_yf_4(x(s),y(t)).
\end{equation*}
\item The coefficients $f_{i3},\ i=1,\ldots,8$ can be also expressed as
\begin{align*}
f_{83}(y(t))&= \mathbb{S}_y(f_{81}(y(t)))+\frac{1}{2}\mathbb{D}_y(f_{81}(y(t)))+\mathbb{D}_y(f_{61}(y(t))),\\
f_{73}(x(s))&=f_{71}(x(s))+\mathbb{D}_y(f_{41}(x(s),y(t))),\\
f_{63}(y(t))&= \mathbb{D}_y(f_{81}(y(t)))V_{2}(t)+\frac{1}{2}\mathbb{S}_y(f_{81}(y(t)))+\mathbb{S}_y(f_{61}(y(t))),\\
f_{53}(x(s))&=f_{51}(x(s))+\mathbb{D}_y(f_{31}(x(s),y(t))), \\
f_{43}(x(s),y(t))&=\frac{1}{2}\mathbb{D}_y(f_{41}(x(s),y(t)))+\mathbb{D}_y(f_{21}(x(s),y(t)))+ \mathbb{S}_y(f_{41}(x(s),y(t))), \\
f_{33}(x(s),y(t))&=\frac{1}{2}\mathbb{D}_y(f_{31}(x(s),y(t)))+\mathbb{D}_y(f_{11}(x(s),y(t)))+ \mathbb{S}_y(f_{31}(x(s),y(t))),\\
f_{23}(x(s),y(t))&=\frac{1}{2}\mathbb{S}_y(f_{41}(x(s),y(t)))+\mathbb{S}_y(f_{21}(x(s),y(t)))+\mathbb{D}_y(f_{41}(x(s),y(t)))V_{2}(t),\\
f_{13}(x(s),y(t))&=\frac{1}{2}\mathbb{S}_y(f_{31}(x(s),y(t)))+\mathbb{S}_y(f_{11}(x(s),y(t)))+\mathbb{D}_y(f_{31}(x(s),y(t)))V_{2}(t),
\end{align*}
where $V_{2}(t)$ is given in \eqref{eq:uuvv}.
\end{enumerate}
\end{remark}

\subsection{Conjecture on the partial difference equation satisfied by the $p$-variate Racah polynomials}\label{sec:conjracah}

The following question arises naturally: can we give the general form and the order of the partial divided-difference equation satisfied by any $p$-variate Racah polynomials in terms of the operators $\mathbb{S}_x$ and $\mathbb{D}_x$? To answer this question, we need to recall some definitions.

Let $N,p$ be two positive integers, $\pmb{x}=(x_1,\ldots,x_p)$, $x_0=0$, $x_{p+1}=N$. The $p$-variate Racah polynomials are defined by \cite[Equation (3.10)]{MR2784425}
\begin{equation}\label{eq:mvrp}
R_{\pmb{n}}({\pmb{x}};{\pmb{\beta}};N)=\prod_{k=1}^p r_{n_k}(2N_1^{k-1}+\beta_k-\beta_0-1,\beta{k+1}-\beta_k-1,N_1^{k-1}-x_{k+1}-1,N_1^{k-1}+\beta_k+x_{k+1};-N_1^{k-1}+x_k),
\end{equation}
where ${\pmb{\beta}}=(\beta_0,\beta_1,\ldots,\beta_{p+1})$, ${\pmb{n}}=(n_1,n_2,\ldots,n_p)\in\mathbb{N}_0^p$ is such that $n_1+n_2+\ldots+n_p\leq N$, and $N_1^j=n_1+n_2+\ldots+n_j$ with $N_1^0=0$. $R_{\pmb{n}}({\pmb{x}};{\pmb{\beta}};N)$ is a polynomial in the lattices
\begin{equation}\label{eq:multilattices}
y_i(x_i)=x_i(x_i+\beta_i), \quad i=1,2,\ldots,p.
\end{equation}
Let us introduce the following notations: for any $l_1,l_2,\ldots,l_p\in\{0,1,2\}$, we define the operator $E_{(l_1,l_2,\ldots,l_p)}$ which is equal to the product of $\mathbb{S}_{y_i}\mathbb{D}_{y_i}$ and $\mathbb{D}_{y_i}^2$ such that for $i=1,2,\ldots,p$,  if $l_i=0$, there is not $\mathbb{S}_{y_i}\mathbb{D}_{y_i}$ and $\mathbb{D}_{y_i}^2$ in the product, if $l_i=1$, then there is $\mathbb{S}_{y_i}\mathbb{D}_{y_i}$ but not $\mathbb{D}_{y_i}^2$ in the product and if $l_i=2$, then there is $\mathbb{D}_{y_i}^2$ but not $\mathbb{S}_{y_i}\mathbb{D}_{y_i}$ in the product. This operator is defined explicitly by 
 \[E_{(l_1,l_2,\ldots,l_p)}=\prod_{i=0}^p\Big(\mathbb{S}_{y_i}\mathbb{D}_{y_i}\Big)^{-l_i(l_i-2)}\Big(\mathbb{D}_{y_i}^2\Big)^{\frac{1}{2}l_i(l_i-1)}.\]
 For example, for $p=3$, $E_{(0,1,0)}=\mathbb{S}_{y_2}\mathbb{D}_{y_2}$, $E_{(0,1,1)}=\mathbb{S}_{y_2}\mathbb{D}_{y_2}\mathbb{S}_{y_3}\mathbb{D}_{y_3}$, $E_{(2,1,0)}=\mathbb{S}_{y_2}\mathbb{D}_{y_2}\mathbb{D}_{y_1}^2$, $E_{(2,2,2)}=\mathbb{D}_{y_1}^2\mathbb{D}_{y_2}^2\mathbb{D}_{y_3}^2$, \ldots.
Using these notations, we can write for $p=1$ Equation \eqref{eqdiv-diff} as
\[
\phi(\eta(s)) E_{(2)} r_{n}(s)+ \tau(\eta(s)) E_{(1)} r_{n}(s)+ \lambda_{n} r_{n}(s)=0,
\]
where $E_{(2)}= {\mathbb{D}}_{\eta}^{2}$ and $E_{(1)}={\mathbb{S}}_{\eta} {\mathbb{D}}_{\eta}$. We can also write for $p=2$ Equation \eqref{eqdiv-diff-biv} as
\begin{multline}\label{eqmultRacGen}
f_1(x(s),y(t)) E_{(2,2)} R_{n,m}(s,t) + f_2(x(s),y(t))E_{(1,2)} R_{n,m}(s,t)
+f_3(x(s),y(t))E_{(2,1)}  R_{n,m}(s,t)\\
+f_4(x(s),y(t)) E_{(1,1)}  R_{n,m}(s,t)
+ f_5(x(s))E_{(2,0)}  R_{n,m}(s,t)+f_6(y(t))E_{(0,2)} R_{n,m}(s,t)+f_7(x(s)) E_{(1,0)} R_{n,m}(s,t)\\
+f_8(y(t))E_{(0,1)} R_{n,m}(s,t)+(m+n)(\beta_{3}-\beta_{0}+m+n-1) R_{n,m}(s,t)=0.
\end{multline}
We have the following conjecture
\begin{conjecture}
The $p$-variate Racah polynomials $R_{\pmb{n}}({\pmb{x}};{\pmb{\beta}};N)$ defined in \eqref{eq:mvrp} are solution of a $2p$-order partial linear divided-difference equation with polynomial coefficients $f_i(\pmb{x})$ of the form
\begin{multline}\label{conjectRacah}
\sum_{\underset{l_1+l_2+\cdots+l_p=i}{i=1}}^{2p} f_i(\pmb{x})E_{(l_1,l_2,\ldots,l_p)} R_{\pmb{n}}({\pmb{x}};{\pmb{\beta}};N)\\
+(n_1+n_2+\cdots+n_p)(\beta_{p+1}-\beta_0+n_1+n_2+\cdots+n_p-1) R_{\pmb{n}}({\pmb{x}};{\pmb{\beta}};N)=0,
\end{multline}
 where $f_i(\pmb{x})$ is a polynomial of degree $l_1+l_2+\cdots+l_p$ in the lattices $y_i(x_i)$, $i=1,\dots,p$, defined in \eqref{eq:multilattices} and if $l_j=0$, then $f_i$ does not depend on $x_j$.
 \end{conjecture}
 \begin{remark}
 \begin{enumerate}
 \item Note that the eigenvalue $-(n_1+n_2+\cdots+n_p)(\beta_{p+1}-\beta_0+n_1+n_2+\cdots+n_p-1)$ is given in \cite[Theorem 3.6]{MR2784425}.
 \item For any $i$ from 1 to $2p$, we take all the combinations of $l_1,l_2,\ldots,l_p\in\{0,1,2\}$ such that $l_1+l_2+\cdots+l_p=i$. It follows that Equation \eqref{eqmultRacGen} has $3^p$ polynomials coefficients since it is supported on the cube $\{0,1,2\}^p$.
 \item Starting from $i=0$ and ending at $i=2p$, for each combination of $l_1,l_2,\ldots,l_p\in\{0,1,2\}$ such that $l_1+l_2+\cdots+l_p=i$, we substitute $R_{\pmb{l}}({\pmb{x}};{\pmb{\beta}};N)$ in \eqref{conjectRacah} to get the coefficient $f_i(x)$.
\item Equation \eqref{conjectRacah} is the analogue of Equation (4.14a) given in \cite[Theorem 4.6]{MR2784425}.
 \end{enumerate}
 \end{remark}

\section{Three-term recurrence relations for bivariate orthogonal polynomial solutions of \eqref{eqdiv-diff-biv}}\label{sec:3}

{}From the partial divided-difference equation \eqref{eqdiv-diff-biv} satisfied by bivariate Racah polynomials we first derive the matrix coefficients in the three-term recurrence relations satisfied by any bivariate orthogonal polynomial solution of the equation. We give explicitly the recurrences satisfied by those families of bivariate Racah polynomials defined in \eqref{eq:brp} and \eqref{eq:secondfamilyracah}. The family of monic bivariate Racah polynomials is introduced from the three-term recurrence relations it obeys, by following a similar approach as already considered in the continuous case \cite{MR2853206}, discrete case \cite{AGR2012} and their $q$-analogues \cite{AAGR2013}. Moreover, by using these results we explicitly solve the connection problem between bivariate Racah polynomials \eqref{eq:brp} and \eqref{eq:secondfamilyracah}.

\medskip
Let $x(s)$ and $y(t)$ be the quadratic lattices defined in \eqref{eq:lattices} and let us denote $\textbf{x}=(x(s),y(t))$ and
$\textbf{x}^n$ ($n\in \mathbb{N}_0$) the column vector of the monomials $x(s)^{n-k} y(t)^{k}$, whose elements are arranged in
graded lexicographical order (see \cite[p. 32]{MR1827871}):
\begin{equation*}
\textbf{x}^n= (x(s)^{n-k}y(t)^{k})\,,\quad 0 \leq k \leq n, \quad
n\in \mathbb{N}_0\,.
\end{equation*}
Let $\{P_{n-k,k}(x(s),y(t))\}$ be a sequence of polynomials in the space $\Pi_n^2$ of all polynomials of total degree at most $n$
in two variables, $\textbf{x}=(x(s),y(t))$, with real coefficients satisfying \eqref{eqdiv-diff-biv}. These polynomials can be expressed as finite sum of terms of the form $ax(s)^{n-k}y(t)^{k}$, where $a \in \mathbb{R}$.

Let ${\mathbf{P}}_n$ denote the (column) polynomial vector of the polynomials $P_{n-k,k}(x(s),y(t))$ of total degree $n$,
\begin{equation*}
{\mathbf{P}}_n= (P_{n,0}(x(s),y(t)), P_{n-1,1}(x(s),y(t)),\dots,P_{1,n-1}(x(s),y(t)), P_{0,n}(x(s),y(t)))^{T}.
\end{equation*}
Then,
\begin{equation}\label{EXPP}
{\mathbf{P}}_n= G_{n,n}\textbf{x}^n+ G_{n,n-1}\textbf{x}^{n-1}+ G_{n,n-2}\textbf{x}^{n-2}+\cdots + G_{n,0}\,\textbf{x}^0,
\end{equation}
where $G_{n,j}$ are matrices of size $(n+1)\times(j+1)$ and $G_{n,n}$ is a nonsingular square matrix of size $(n+1)\times(n+1)$.

Following \cite{FKKM}, let us introduce the following bases $\{F_n(x(s))\}_{n\in\mathbb N}$ and $\{F_n(y(t))\}_{n\in\mathbb N}$ of monic polynomials in the quadratic lattices $x(s)$ and $y(t)$ defined in \eqref{eq:lattices}
\begin{align}
F_n(x(s))=(-4)^{-n} \left(-\beta_{1}-2 s+\frac{1}{2}\right)_n \left(\beta_{1}+2 s+\frac{1}{2}\right)_n, \label{eq:baseff1} \\
F_n(y(t))=(-4)^{-n} \left(-\beta_{2}-2 t+\frac{1}{2}\right)_n \left(\beta_{2}+2 t+\frac{1}{2}\right)_n \label{eq:baseff2},
\end{align}
where we have used the Pochhammer symbol. If we denote the column vector
\begin{equation*}
\textbf{F}_n= (F_{n-k}(x(s))F_{k}(y(t)))\,,\quad 0 \leq k \leq n, \quad n\in \mathbb{N}_0\,,
\end{equation*}
we can also write
\begin{equation}\label{EXPP1}
{\mathbf{P}}_n= G_{n,n}'\textbf{F}_n+ G_{n,n-1}'\textbf{F}_{n-1}+ G_{n,n-2}'\textbf{F}_{n-2}+\dots + G_{n,0}'\,\textbf{F}_0,
\end{equation}
where $G_{n,j}'$ are matrices of size $(n+1)\times(j+1)$ and $G_{n,n}'$ is a nonsingular square matrix of size $(n+1)\times(n+1)$.

Next, we shall give explicit expressions for the matrices $A_{n,j}$ of size $(n+1) \times (n+2)$, $B_{n,j}$ of size $(n+1) \times (n+1)$, and $C_{n,j}$ of size $(n+1) \times n$ appearing in the three-term recurrence relations
\begin{equation}\label{RRTT}
x_j{\mathbf{P}}_n=A_{n,j}{\mathbf{P}}_{n+1} + B_{n,j}{\mathbf{P}}_{n} + C_{n,j}{\mathbf{P}}_{n-1}, \quad j =1, 2,
\end{equation}
with the initial conditions ${\mathbf{P}}_{-1}=0$ and ${\mathbf{P}}_{0}=1$, where we have used the notations $x_{1}=x(s)$ and
$x_{2}=y(t)$ as well as \eqref{eq:lattices}, in terms of the polynomial coefficients of the fourth-order linear partial divided-difference equation \eqref{eqdiv-diff-biv}. As a consequence, we shall obtain the matrices $A_{n,j}$, $B_{n,j}$ and $C_{n,j}$ for both families of bivariate Racah polynomials \eqref{eq:brp} and \eqref{eq:secondfamilyracah}. Moreover, we shall introduce the family of monic bivariate Racah polynomials also solution of \eqref{eqdiv-diff-biv}. In doing so, we shall need the following properties \cite{FKKM} of the bases $\{F_n(x(s))\}_{n\in\mathbb N}$ and $\{F_n(y(t))\}_{n\in\mathbb N}$ defined in \eqref{eq:baseff1} and \eqref{eq:baseff2} respectively,
\begin{align}
&\mathbb D_x F_{n}(x(s))=n F_{n-1}(x(s)), \quad \mathbb D_y F_{n}(y(t))=n F_{n-1}(y(t)), \label{1}\\
& x(s)F_n(x(s))=F_{n+1}(x(s))+f_{n}(\beta_{1})\,F_{n}(x(s)),\quad y(t)F_n(y(t))=F_{n+1}(y(t))+f_{n}(\beta_{2})\,F_{n}(y(t)), \label{3} \\
&\mathbb S_x F_{n}(x(s))=F_{n}(x(s))+g_{n} \,F_{n-1}(x(s)), \quad S_y F_{n}(y(t))=F_{n}(y(t))+g_{n} \,F_{n-1}(y(t)), \label{2}
\end{align}
with
\begin{equation}\label{eq:gammafg}
f_{n}(\beta_{i})=\frac{1}{16} \left((2 n+1)^2-4 \beta_{i}^2\right), \,\,i=1,2, \qquad g_{n}=\frac{1}{4} n (2 n-1).
\end{equation}

{}From \eqref{1}, \eqref{3}, and \eqref{2} we obtain the following identities of column matrices
\begin{equation*}
\begin{cases}
\mathbb D_x \textbf{F}_n=E_{n,1}\textbf{F}_{n-1}, \quad
\mathbb D_y \textbf{F}_n=E_{n,2}\textbf{F}_{n-1}, \\
\mathbb S_x \textbf{F}_n=\textbf{F}_n+J_{n,1}\textbf{F}_{n-1}, \quad
\mathbb S_y \textbf{F}_n=\textbf{F}_n+J_{n,2}\textbf{F}_{n-1}, \\
x(s)\textbf{F}_n=L_{n,1}\textbf{F}_{n+1}+M_{n,1}\textbf{F}_n,\quad
y(t)\textbf{F}_n=L_{n,2}\textbf{F}_{n+1}+M_{n,2}\textbf{F}_n,
\end{cases}
\end{equation*}
where $E_{n,1}$ and $E_{n,2}$ are the matrices of size $(n+1)\times n$ given by
\[
E_{n,1}=
\begin{pmatrix}
  n & 0 & \ldots & 0 \\
  0 & {n-1} & \ddots & \vdots \\
  \vdots & \ddots & \ddots & 0 \\
  \vdots &  & \ddots & 1 \\
  0 & \ldots & \ldots & 0 \\
\end{pmatrix}, \qquad
E_{n,2}=
\begin{pmatrix}
  0 & \ldots & \ldots & 0 \\
  1 & \ddots &  & \vdots \\
  0 & 2 & \ddots & \vdots \\
  \vdots & \ddots & \ddots & 0 \\
  0 & \ldots & 0 & n \\
\end{pmatrix},
\]
$J_{n,1}$ and $J_{n,2}$ are the matrices of size $(n+1)\times n$ given by
\begin{equation*}
J_{n,1}=
\begin{pmatrix}
  g_{n} & 0 & \ldots & 0 \\
  0 & g_{n-1} & \ddots & \vdots \\
  \vdots & \ddots & \ddots & 0 \\
  \vdots &  & \ddots & g_{1} \\
  0 & \ldots & \ldots & 0 \\
\end{pmatrix}, \qquad
J_{n,2}=
\begin{pmatrix}
  0 & \ldots & \ldots & 0 \\
g_{1} & \ddots &  & \vdots \\
  0 & g_{2} & \ddots & \vdots \\
  \vdots & \ddots & \ddots & 0 \\
  0 & \ldots & 0 & g_{n} \\
\end{pmatrix},
\end{equation*}
where $g_{n}$ are given in \eqref{eq:gammafg}, $L_{n,1}$ and $L_{n,2}$ are the matrices of size $(n+1)\times(n+2)$
\begin{equation}\label{eq:llmatrices}
L_{n,1}=
\begin{pmatrix}
  1 & 0 & \ldots & \ldots & 0 \\
  0 & 1 & \ddots &  & \vdots \\
  \vdots & \ddots & \ddots & \ddots & \vdots \\
  0 & \ldots & 0 & 1 & 0 \\
\end{pmatrix},
\qquad L_{n,2}=
\begin{pmatrix}
  0 & 1 & 0 & \ldots & 0 \\
  \vdots & \ddots & 1 & \ddots & \vdots \\
  \vdots &  & \ddots & \ddots & 0 \\
  0 & \ldots & \ldots & 0 & 1 \\
\end{pmatrix},
\end{equation}
and $M_{n,1}$ and $M_{n,2}$ are the matrices of size $(n+1)\times(n+1)$
\[
M_{n,1}=
\begin{pmatrix}
f_{n}(\beta_{1}) & 0 & \ldots & 0 \\
  0 & f_{n-1}(\beta_{1}) & \ddots & \vdots \\
  \vdots & \ddots & \ddots & 0 \\
  0 & \ldots & 0 & f_{0}(\beta_{1}) \\
\end{pmatrix}, \qquad
M_{n,2}=
\begin{pmatrix}
f_{0}(\beta_{2}) & 0 & \ldots & 0 \\
  0 & f_{1}(\beta_{2}) & \ddots & \vdots \\
  \vdots & \ddots & \ddots & 0 \\
  0 & \ldots & 0 & f_{n}(\beta_{2}) \\
\end{pmatrix},
\]
where $f_{n}(\beta_{1})$ and $f_{n}(\beta_{2})$ are given in \eqref{eq:gammafg}.

\medskip
If we substitute the expansion \eqref{EXPP1} in \eqref{eqdiv-diff-biv}, by equating the coefficients in  $\textbf{F}_{n-1}$ and $\textbf{F}_{n-2}$ we obtain the following explicit expressions for the matrices ${G}_{n,n-1}'$ and ${G}_{n,n-2}'$:
\begin{equation}\label{eq:matricesgp}
\begin{cases}
G_{n,n-1}' &= G_{n,n}'\mathbf{S}_n {\mathbf{Z}}_{n-1}^{-1}(\lambda_n), \\
G_{n,n-2}' &= (G_{n,n}'{\mathbf{T}}_n+G_{n,n-1}'{\mathbf{S}}_{n-1}) {\mathbf{Z}}_{n-2}^{-1}(\lambda_n),
\end{cases}
\end{equation}
in terms of the nonsingular matrices $G_{n,n}'$, where
\[
{\mathbf{Z}}_n(\lambda_{\ell})=(\lambda_n-\lambda_{\ell}) {\mathbf{I}}_{n+1},
\]
${\mathbf{I}}_{n}$ stands for the identity matrix of size $n$, and $\lambda_{n}=n(\beta_{3}-\beta_{0}+n-1) $.

The matrix ${\mathbf{S}}_{n}$ of size $(n+1)\times n$ is given in terms of the polynomial coefficients of the equation \eqref{eqdiv-diff-biv} as
\begin{equation}\label{eq:sn}
{\mathbf{S}}_n=\left(%
\begin{array}{cccc}
  s_{1,1} & 0 & \ldots & 0 \\
  s_{2,1} & s_{2,2} & \ddots & \vdots \\
  0 & s_{3,2} & \ddots & 0 \\
  \vdots & \ddots & \ddots & s_{n,n} \\
  0 & \ldots & 0 & s_{n+1,n} \\
\end{array}%
\right) \qquad (n\geq 1),
\end{equation}
where
\begin{align*}
s_{k,k}&=-\frac{1}{16} (k-n-1) \left(-\beta _3+8 \beta _1 \left(2 \left(k+n-k n+N^2-1\right)+\beta _1 (n-1)\right)+\beta _0 \left(-4 \beta _1^2+8 \beta _1 (k-n) \right. \right. \\
& \left. \left. -4 (k-3 n) (k+n)+16 \beta _3 (n-N-1)-32 n-16 N^2+17\right)+16 N^2 (n-k)+4 \beta _3 \left(\beta _1-k+n\right) \right. \\
& \left. \times \left(\beta _1-k-3 n+4 N+4\right)-2 (n-1) (-4 (k-2) k+4 (n-2) n+1)\right), \qquad  k=1,\ldots,n,\\
s_{k+1,k}&=\frac{1}{16} k \left(-13 \beta _3+4 \beta _3 \left(k^2+\beta _2 \left(\beta _2-2 k+4 N+2\right)-2 k (2 N+1)-4 \left(n^2-2 n (N+1)+N\right)\right) \right. \\
& \left. +8 \beta _2 \left(2 \left(-k n+k+(n-1) n+N^2\right)+\beta _2 (n-1)\right)+\beta _0 \left(-4 \beta _2^2+8 \beta _2 (k-2 n+1) \right. \right. \\
& \left. \left. -4 k (k-4 n+2)+16 \beta _3 (n-N-1)-16 n-16 N^2+13\right)-16 N^2 (k-2 n+1) \right. \\
& \left. +2 (n-1) (4 k (k-2 n)+8 n-5)\right), \qquad k=1,\ldots,n.
\end{align*}

Moreover, the matrix $\mathbf{T}_n$ of size $(n+1)\times (n-1)$ is given in terms of the polynomial coefficients of the equation
\eqref{eqdiv-diff-biv} as
\begin{equation}\label{eq:tn}
\mathbf{T}_n=\left(%
\begin{array}{ccccc}
  t_{1,1} & 0 & \cdots & \cdots & 0 \\
  t_{2,1} & t_{2,2} & \ddots &  & \vdots \\
  t_{3,1} & t_{3,2} & \ddots & \ddots & \vdots \\
  0 & t_{4,2} & \ddots & \ddots & 0 \\
  \vdots & \ddots & \ddots & \ddots & t_{n-1,n-1} \\
  \vdots &  & \ddots & \ddots & t_{n,n-1} \\
  0 & \cdots & \cdots & 0 & t_{n+1,n-1} \\
\end{array}%
\right) \qquad (n\geq 2),
\end{equation}
where, for $1\leq k\leq n-1$,
\begin{align*}
t_{k,k}&=-\frac{1}{256} (n-k) (n-k+1) \left(-2 \beta _1+2 k-2 n+1\right) \left(4 \beta _0-2 \beta _1+2 k-2 n+1\right) \\
& \times \left(-2 \beta _1+2 k+2 n-4 N-5\right) \left(-2 \beta _1+4 \beta _3+2 k+2 n+4 N-5\right),\\
t_{k+2,k}&=-\frac{1}{256} k (k+1) \left(-2 \beta _2+2 k-4 n+5\right) \left(4 \beta _0-2 \beta _2+2 k-4 n+5\right) \left(-2 \beta _2+2 k-4 N-1\right) \\
&\times \left(-2 \beta _2+4 \beta _3+2 k+4 N-1\right),  \\
t_{k+1,k}&=\frac{1}{128} k (k-n) \left(-160 \beta _3+8 \beta _1 \left(8 \beta _2 \left(k^2-k n+\beta _3 (k-2 N-1)-2 N^2+1\right)-4 \beta _2^2 \left(\beta _3+k+n \right. \right. \right. \\
& \left. \left. \left. -2\right)+\beta _3  (-4 k (k-4 n+4)-16 n N-8 n+16 N+5)+k (-4 k (k-3 n+2)-16 n+13) \right. \right. \\
& \left. \left. -16 n N^2+5 n+16 N^2-2\right)+4 \beta _2 \left(2 \left(-4 k^3+8 k^2+k \left(4 n (3 n-8)+16
   N^2+13\right)-2 (n-1)   \right. \right. \right. \\
& \left. \left. \left. \times \left(4 (n-2) n+8 N^2+1\right)\right)+\beta _2 \left(8 k n+4 (k-4) k-12 n^2+32 n-21\right)\right)+8 \beta _0 \left(-26 \beta _3+4 k^3 \right. \right.  \\
& \left. \left.  +\beta _2   \left(-4 k^2 -8 k (n-2)+4 \beta _2 (2 n-3)+12 (n-2) n+16 N^2+5\right)+4 \beta _3 \left(2 k^2-2 k (n+2 N) \right. \right. \right. \\
& \left. \left. \left. +\beta _2 \left(\beta _2-2 k+4 N+2\right)+n (8 N-3 n+10)-8 N\right)-8 k^2+4 \beta _1^2 \left(\beta _3-\beta _2+k-1\right)-12 k n^2 \right. \right. \\
& \left. \left. -8 \beta _1 \left(-\beta _2+\beta _3+k-1\right) (k-n+1)-16 N^2 (k-2 n+2)+32 k n-13 k+12 n^2-34n+20\right) \right.   
\end{align*}
\begin{align*}
& \left. +8 \beta _3 \left(4 k^3-4 k^2 (3 n+2 N-2)-4 \beta _2 (k-n+1) \left(-\beta _2+2 k-4 N-2\right)+k (32 (n-1) N+16 n \right. \right.  \\
& \left. \left.-13)+n (4 n (2 n-6 N-9)+48 N+47)-26 N\right)+4 \left(4 k^4-8 k^3 n-2 k^2 \left(6 (n-4) n+8 N^2 \right. \right. \right. \\
& \left. \left. \left. +13\right)+2 k \left(32 (n-1) N^2+n (8 (n-3) n+13)\right)+n \left((63-16 n) n-48 (n-2) N^2\right)\right) \right. \\
& \left.+4 \beta _1^2 \left(4 \beta _2^2-8 \beta _2 (k-2 n+2)+4 k (k-4 n+4)+8 \beta _3 (-2 n+2 N+3)+16 n+16 N^2-21\right) \right. \\
& \left. -304 n-208 N^2+121\right).
\end{align*}

In order to obtain the coefficients $G_{n,n-1}$ and $G_{n,n-2}$ of $\mathbf P_n$ in \eqref{EXPP}, we use the following relations:
\begin{align}
    F_n(x(s))&=x(s)^n+H_{n,n-1}^{(1)}x(s)^{n-1}+H_{n,n-2}^{(1)}x(s)^{n-2}+\textrm{terms of lower degree}, \label{eq:rel1}\\
        F_n(y(t))&=y(t)^n+H_{n,n-1}^{(2)}y(t)^{n-1}+H_{n,n-2}^{(2)}y(t)^{n-2}+\textrm{terms of lower degree},\label{eq:rel2}
\end{align}
where
\begin{equation}\label{eq:hhhh}
\begin{cases}
H_{n,n-1}^{(1)}=\displaystyle{\frac{1}{48} \left(-4 n^3+12 \beta _1^2 n+n\right)}, \\[2mm]
H_{n,n-2}^{(1)}=\displaystyle{\frac{(n-1) n \left(720 \beta _1^4+120 \beta _1^2 \left(1-4 n^2\right)+(2 n-3) (2 n-1) (2 n+1) (10 n+7)\right)}{23040}}, \\[2mm]
H_{n,n-1}^{(2)}=\displaystyle{\frac{1}{48} \left(-4 n^3+12 \beta _2^2 n+n\right)}, \\[2mm]
H_{n,n-2}^{(2)}=\displaystyle{\frac{(n-1) n \left(720 \beta _2^4+120 \beta _2^2 \left(1-4 n^2\right)+(2 n-3) (2 n-1) (2 n+1) (10 n+7)\right)}{23040}},
\end{cases}
\end{equation}
to write
\begin{eqnarray}\label{4}
    \textbf{F}_n=\textbf{x}^n+U_{n,n-1}\textbf{x}^{n-1}+U_{n,n-2}\textbf{x}^{n-2}+\textrm{terms of lower degree}
\end{eqnarray}
where $U_{n,n-1}$ is the matrix of size $(n+1)\times n$
\[
U_{n,n-1}=
\begin{pmatrix}
  H_{n,n-1}^{(1)} & 0 & \ldots & 0 \\
  H_{1,0}^{(2)} & H_{n-1,n-2}^{(1)} & \ddots & \vdots \\
  0 & H_{2,1}^{(2)} & \ddots & 0 \\
  \vdots & \ddots & \ddots & H_{1,0}^{(1)} \\
  0 & \ldots & 0 & H_{n,n-1}^{(2)}
\end{pmatrix},
\]
and $U_{n,n-2}$ is the matrix of size $(n+1)\times(n-1)$
\[
U_{n,n-2}=
\begin{pmatrix}
  H_{n,n-2}^{(1)} & 0 & \cdots & \cdots & 0 \\
  H_{n-1,n-2}^{(1)}H_{1,0}^{(2)} & H_{n-1,n-3}^{(1)} & \ddots &  & \vdots \\
  H_{2,0}^{(2)} & H_{n-2,n-3}^{(1)}H_{2,1}^{(2)} & \ddots & \ddots & \vdots \\
  0 & H_{3,1}^{(2)} & \ddots & \ddots & 0 \\
  \vdots & \ddots & \ddots & \ddots & H_{2,0}^{(1)} \\
  \vdots &  & \ddots & \ddots & H_{1,0}^{(1)}H_{n-1,n-2}^{(2)} \\
  0 & \cdots & \cdots & 0 & H_{n,n-2}^{(2)}
\end{pmatrix}.
\]
By replacing \eqref{4} into \eqref{EXPP1} and identifying to \eqref{EXPP} it yields
\begin{equation}\label{eq:matricesgn}
\begin{cases}
G_{n,n}=G_{n,n}' ,\\
G_{n,n-1}= G_{n,n}'U_{n,n-1}+G_{n,n-1}' ,\\
G_{n,n-2}= G_{n,n}'U_{n,n-2}+G_{n,n-1}'U_{n-1,n-2}+G_{n,n-2}',
\end{cases}
\end{equation}
where the matrices $G_{n,n-1}'$ and $G_{n,n-2}'$ are given in \eqref{eq:matricesgp} in terms of the leading coefficient $G_{n,n}'$ of ${\mathbf{P}}_{n}$ \eqref{EXPP1}.

As a consequence, we obtain
\begin{theorem}\label{eq:matricesTTRR}
The explicit expressions of the matrices $A_{n,j}$, $B_{n,j}$ and $C_{n,j}$ ($j=1,2$) in the three-term recurrence relations
\eqref{RRTT} can be expressed in terms of the coefficients of \eqref{eqdiv-diff-biv} as
\begin{align*}
A_{n,j}&=G_{n,n}L_{n,j}G_{n+1,n+1}^{-1}, \quad n \geq 0,\\
B_{0,j}&=\left(-A_{0,j} G_{1,0}\right)G_{0,0}^{-1},  \\
B_{n,j}&=\left(G_{n,n-1}L_{n-1,i}-A_{n,j} G_{n+1,n}\right)G_{n,n}^{-1}, \quad n \geq 1,\\
C_{1,j}&=\left(-A_{1,j}G_{2,0}-B_{1,j}G_{1,0}\right)G_{0,0}^{-1}, \\
C_{n,j}&=\left(G_{n,n-2}L_{n-2,j}-A_{n,j}G_{n+1,n-1}-B_{n,j}G_{n,n-1}\right)G_{n-1,n-1}^{-1}, \quad n \geq 2,
\end{align*}
where the matrices $G_{n,n-1}$ and $G_{n,n-2}$ are given in \eqref{eq:matricesgn}.
\end{theorem}

Note that if we replace $G_{n,n}$ by any invertible matrix of size $(n+1)\times(n+1)$ we obtain a polynomial solution of \eqref{eqdiv-diff-biv}.

If we choose $G_{n,n}$ as the identity matrix in Theorem \ref{eq:matricesTTRR}, we obtain
\begin{corollary}\label{eq:matricesTTRRmonic}
The explicit expressions of the matrices $A_{n,j}$, $B_{n,j}$ and $C_{n,j}$ ($j=1,2$) in the three-term recurrence relations
\eqref{RRTT} in the case of monic polynomial solutions $\hat{{\mathbf{P}}}_{n}$ of the fourth-order linear partial divided-difference equation \eqref{eqdiv-diff-biv} can be expressed in terms of the coefficients of the equation as
\begin{align*}
A_{n,j}&=L_{n,j}, \quad n \geq 0,\\
B_{0,j}&=-L_{0,j} G_{1,0}, \qquad \qquad \qquad \,
B_{n,j}=G_{n,n-1}L_{n-1,i}-L_{n,j} G_{n+1,n}, \quad n \geq 1,\\
C_{1,j}&=-L_{1,j}G_{2,0}-B_{1,j}G_{1,0}, \qquad
C_{n,j}=G_{n,n-2}L_{n-2,j}-L_{n,j}G_{n+1,n-1}-B_{n,j}G_{n,n-1}, \quad n \geq 2,
\end{align*}
where the matrices $G_{n,n-1}$ and $G_{n,n-2}$ are given in \eqref{eq:matricesgn}.
\end{corollary}

Let us assume that the matrices in the three-term recurrence relations satisfy the rank conditions $\text{rank}\,A_{n,j}=\text{rank}\,C_{n+1,j}=n+1$ ($j=1,2$), for the joint matrix $A_{n}$ of $A_{n,1}$ and $A_{n,2}$, $A_{n}=(A_{n,1}^{\mathsf T},A_{n,2}^{\mathsf T})^{\mathsf T}$ we have $\text{rank}\,A_{n}=n+2$, and for the joint matrix $C_{n+1}$ of $C_{n+1,1}$ and $C_{n+1,2}$, $C_{n+1}^{\mathsf T}=(C_{n,1}^{\mathsf T},C_{n,2}^{\mathsf T})$, we have $\text{rank}\,C_{n+1}^{\mathsf T}=n+2$. Since for $n \geq 0$ there exist matrices $A_{n,j}$, $B_{n,j}$ and $C_{n,j}$ ($j=1,2$) such that the column vector of polynomials ${\mathbf{P}}_{n}$ satisfy the three-term recurrence relations \eqref{RRTT}, then applying \cite[Theorem 3.2.7]{MR1827871} we have that there exists a linear functional $\mathcal{L}$ which defines a quasi-definite linear functional and which makes $\{ {\mathbf{P}}_{n}\}_{n=0}^{\infty}$ an orthogonal basis in the space of bivariate polynomials.

\subsection{Explicit expressions of the leading matrices $G_{n,n}$ for bivariate Racah polynomials \eqref{eq:brp} and \eqref{eq:secondfamilyracah}}
In order to obtain the matrix $G_{n,n}$ for the particular case of both families of bivariate Racah polynomials $R_{n,m}(s,t;\beta_{0},\beta_{1},\beta_{2},\beta_{3},N)$  and $\bar{R}_{n,m}(s,t;\beta_{0},\beta_{1},\beta_{2},\beta_{3},N)$  we use the connection formula
\begin{equation*}
(-s)_{n}(s+\beta_{1})_{n}
=\sum_{j=0}^{n} \,\frac{(-1)^j\, 2^{2 j-2 n} \,(n-j+1)_j \,\left(\beta_{1}+j-\frac{1}{2}\right)_{2 n-2 j}}{j!} \,F_j(x(s)),
\end{equation*}
as well as their definitions in \eqref{eq:brp} and \eqref{eq:secondfamilyracah}, respectively to obtain the matrices of leading coefficients $G_{n,n}$ in \eqref{EXPP} as
\begin{equation}\label{eq:gnnbr1}
G_{n,n}=G_{n,n}(\beta_{0},\beta_{1},\beta_{2},\beta_{3},N)=
\begin{pmatrix}
g_{i,j}(n,\beta_{0},\beta_{1},\beta_{2},\beta_{3},N)
\end{pmatrix}_{0 \leq i,j \leq n},
\end{equation}
for bivariate Racah polynomials $R_{n,m}(s,t;\beta_{0},\beta_{1},\beta_{2},\beta_{3},N)$ where
\begin{multline*}
g_{i,j}(n,\beta_{0},\beta_{1},\beta_{2},\beta_{3},N)\\
=\begin{cases}
\displaystyle{0}, & i<j, \\[3mm]
\displaystyle{\frac{(-1)^{i+n} \,(\beta_{1}-\beta_{0})_{n-i} \,(2 n-i-\beta_{0}+\beta_{3}-1)_i \, (i-n)_{n-j} \,(n-i-\beta_{0}+\beta_{2}-1)_{n-j}}{(n-j)! \,(\beta_{1}-\beta_{0})_{n-j}}}, & i \geq j,
\end{cases}
\end{multline*}
and
\begin{equation}\label{eq:gnnbr2}
\bar{G}_{n,n}=\bar{G}_{n,n}(\beta_{0},\beta_{1},\beta_{2},\beta_{3},N)=
\begin{pmatrix}
\bar{g}_{i,j}(n,\beta_{0},\beta_{1},\beta_{2},\beta_{3},N)
\end{pmatrix}_{0 \leq i,j \leq n},
\end{equation}
for bivariate Racah polynomials $\bar{R}_{n,m}(s,t;\beta_{0},\beta_{1},\beta_{2},\beta_{3},N)$ where
\begin{multline*}
\bar{g}_{i,j}(n,\beta_{0},\beta_{1},\beta_{2},\beta_{3},N)\\=\begin{cases}
\displaystyle{0}, & i>j, \\[3mm]
\displaystyle{\binom{i}{j}\, \left(\beta _2-\beta _3-i+1\right)_{i-j}\, \left(i-\beta _1+\beta _3-1\right)_j \, \left(i+n-\beta _0+\beta_3-1\right)_{n-i}}, & i \leq j.
\end{cases}
\end{multline*}
By using these matrices $G_{n,n}(\beta_{0},\beta_{1},\beta_{2},\beta_{3},N)$ and $\bar{G}_{n,n}(\beta_{0},\beta_{1},\beta_{2},\beta_{3},N)$ it is possible to apply Theorem \ref{eq:matricesTTRR} in order to compute the families of bivariate Racah polynomials $R_{n,m}(s,t;\beta_{0},\beta_{1},\beta_{2},\beta_{3},N)$ defined in \eqref{eq:brp} and $\bar{R}_{n,m}(s,t;\beta_{0},\beta_{1},\beta_{2},\beta_{3},N)$ defined in \eqref{eq:secondfamilyracah} from the three-term recurrence relations they satisfy.

\subsection{Monic bivariate Racah polynomials and connection between bivariate Racah polynomials \eqref{eq:brp} and \eqref{eq:secondfamilyracah}}\label{section:mbr}

Let us assume that $G_{n,n}$ is the identity matrix of size $n+1$. Then, by using Corollary \ref{eq:matricesTTRRmonic} we introduce the family of monic polynomial solutions of the fourth-order linear partial divided-difference equation \eqref{eqdiv-diff-biv}. We would like to emphasize that in the case of monic polynomial solutions of \eqref{eqdiv-diff-biv}, for each $n$ and $m$, only one term of total degree $n+m$ appears, which moreover has leading coefficient equal to one.

Let
\[
{\mathbf{P}}_n=(R_{n-k,k}(s,t;\beta_{0},\beta_{1},\beta_{2},\beta_{3},N))_{k=0,\dots,n}, \qquad
{\mathbf{\bar{P}}}_n=(\bar{R}_{n-k,k}(s,t;\beta_{0},\beta_{1},\beta_{2},\beta_{3},N))_{k=0,\dots,n},
\]
and let ${\mathbf{\hat{P}}}_n$ be the column vector of monic bivariate Racah polynomials generated from Corollary \ref{eq:matricesTTRRmonic}. Then, we have
\[
{\mathbf{P}}_n= G_{n,n}(\beta_{0},\beta_{1},\beta_{2},\beta_{3},N)\,{\mathbf{\hat{P}}}_n, \qquad
{\mathbf{\bar{P}}}_n= \bar{G}_{n,n}(\beta_{0},\beta_{1},\beta_{2},\beta_{3},N)\,{\mathbf{\hat{P}}}_n.
\]
where $G_{n,n}(\beta_{0},\beta_{1},\beta_{2},\beta_{3},N)$ and $\bar{G}_{n,n}(\beta_{0},\beta_{1},\beta_{2},\beta_{3},N)$ are the matrices of size $n+1$ defined in \eqref{eq:gnnbr1} and \eqref{eq:gnnbr2}, respectively. As a consequence, we obtain the following connection formulae for $n \geq 0$
\begin{equation}\label{eq:connection}
\begin{cases}
{\mathbf{P}}_n=G_{n,n}(\beta_{0},\beta_{1},\beta_{2},\beta_{3},N) (\bar{G}_{n,n}(\beta_{0},\beta_{1},\beta_{2},\beta_{3},N))^{-1}\,{\mathbf{\bar{P}}}_n, \\[3mm]
{\mathbf{\bar{P}}}_n=\bar{G}_{n,n}(\beta_{0},\beta_{1},\beta_{2},\beta_{3},N) (G_{n,n}(\beta_{0},\beta_{1},\beta_{2},\beta_{3},N))^{-1}\,{\mathbf{P}}_n.
\end{cases}
\end{equation}

\section{Fourth-order linear partial divided-difference equation satisfied by the bivariate Wilson, bivariate continuous dual Hahn, and bivariate continuous Hahn polynomials}\label{sec:4}

In this section, we derive the fourth-order linear partial divided-difference equation of the bivariate Wilson polynomials from the bivariate Racah ones. Using limiting process, the partial divided-difference equation of the bivariate continuous dual Hahn and the bivariate continuous Hahn polynomials follow. The coefficients of the three-term recurrence relations satisfied by these families are also given, and the monic bivariate polynomial solutions of the equations are also introduced.

\subsection{Fourth-order linear partial divided-difference equation of the bivariate Wilson polynomials}

If we make the change of variables \cite[p. 443]{MR2784425}
\begin{equation}\label{eq:ractowil}
V_0:=\begin{cases}
&\beta_0=a-b,\quad \beta_1=2a,\quad \beta_2=2a+2e_2,\quad \beta_3=2a+2e_2+c+d, \\
&s=-a+ix, \quad t=-a-e_2+iy, \quad N=-a-d-e_2,
\end{cases}
\end{equation}
we observe that the bivariate Racah polynomials \eqref{eq:brp} transform into the bivariate Wilson polynomials (in a similar way as in the univariate case \cite[p. 196]{MR2656096})
\begin{multline}\label{eq:bwp}
W_{n,m}(x,y;a,b,c,d;e_2)=w_n(x^2;a,b,e_2+iy,e_2-iy)w_m(y^2;n+a+e_2,n+b+e_2,c,d),
\end{multline}
where from now on $x$ and $y$ are the variables and $w_n(x^2;a,b,c,d)$ are the Wilson polynomials defined by \cite[(9.1.1)]{MR2656096}
\begin{multline}\label{eq:wilson}
w_n(x^2;a,b,c,d)=(a+b)_{n}(a+c)_{n}(a+d)_{n}\hyper{4}{3}{-n,n+a+b+c+d-1,a+ix,a-ix}{a+b,a+c,a+d}{1}.
\end{multline}
Note that there is a misprint in \cite[p. 443]{MR2784425} on the change $x=-a+iy$ from the Racah to the Wilson polynomials, and on the change
$x_k=-\varepsilon_2^k-a+iy_k$ from the multivariate Racah to the multivariate Wilson polynomials.\\
The appropriate operators for the Wilson polynomials are the operator $\mathbf{S}_x$ and the Wilson operator $\mathbf{D}_x $ defined by \cite{IsmailStanton2012}
\[
\mathbf{D}_x f(x)=\frac{f\Big(x+\frac{i}{2}\Big)-f\Big(x-\frac{i}{2}\Big)}{2ix},\quad
\mathbf{S}_x f(x)=\frac{f\Big(x+\frac{i}{2}\Big)+f\Big(x-\frac{i}{2}\Big)}{2}.
\]
If we perform the changes \eqref{eq:ractowil}, we obtain by simple computations that
\begin{align*}
& \mathbb{D}_xR_{n,m}(s,t;\beta_{0},\beta_{1},\beta_{2},\beta_3,N)\Big \vert_{V_0}=  -\mathbf{D}_xW_{n,m}(x,y;a,b,c,d;e_2),\\
&\mathbb{D}_yR_{n,m}(s,t;\beta_{0},\beta_{1},\beta_{2},\beta_3,N)\Big \vert_{V_0}=  -\mathbf{D}_yW_{n,m}(x,y;a,b,c,d;e_2), \\
& \mathbb{S}_xR_{n,m}(s,t;\beta_{0},\beta_{1},\beta_{2},\beta_3,N)\Big \vert_{V_0}= \mathbf{S}_xW_{n,m}(x,y;a,b,c,d;e_2),\\
&\mathbb{S}_yR_{n,m}(s,t;\beta_{0},\beta_{1},\beta_{2},\beta_3,N)\Big \vert_{V_0}= \mathbf{S}_yW_{n,m}(x,y;a,b,c,d;e_2).
\end{align*}
It follows from the bivariate Racah divided-difference equation \eqref{eqdiv-diff-biv} and \eqref{eq:ractowil} that
\begin{proposition}\label{prop:divided_biv_Wilson}
The bivariate Wilson polynomials $W_{n,m}(x,y):=W_{n,m}(x,y;a,b,c,d;e_2)$ are solution of the fourth-order linear partial divided-difference equation
 \begin{multline}\label{eqdiv-diff-biv-wil}
f_1(x,y) \mathbf{D}^2_x\mathbf{D}^2_yW_{n,m}(x,y) + f_2(x,y)\mathbf{S}_x\mathbf{D}_x\mathbf{D}^2_yW_{n,m}(x,y)+f_3(x,y)\mathbf{S}_y\mathbf{D}_y\mathbf{D}^2_xW_{n,m}(x,y)\\
+f_4(x,y) \mathbf{S}_x\mathbf{D}_x\mathbf{S}_y\mathbf{D}_y W_{n,m}(x,y)
+ f_5(x)\mathbf{D}^2_x W_{n,m}(x,y)+f_6(y)\mathbf{D}^2_y W_{n,m}(x,y)+f_7(x) \mathbf{S}_x\mathbf{D}_xW_{n,m}(x,y) \\
+f_8(y)\mathbf{S}_y\mathbf{D}_y W_{n,m}(x,y)+(m+n)(2e_2+a+b+c+d+m+n-1) W_{n,m}(x,y)=0,
\end{multline}
where
\begin{align*}
f_8(y)&=\left( -a-b-2 e_2-c-d \right) y^{2}+\left( c+d \right) e_2^{2}+ \left( ad+ca+db+bc+2\,dc \right) e_2\\
&+adc+dba+bac+dbc, \\
f_7(x)&=\left( -a-b-2\,e_2-c-d \right)x^2+\left( a+b \right)e_2^{2}+ \left( bc+db+ad+2\,ba+ca \right)e_2\\
&+bac+dbc+adc+dba,  \\
f_6(y)&=-y^{4}+ \left( be_2+ba+2\,ce_2+ca+ae_2+e_2^{2}+bc+dc+db +2e_2d+ad \right) y^{2}\\
&-dc \left( e_2+b \right) \left( e_2+a \right),  \\
f_5(x)&=-x^{4}+ \left( e_2^{2}+2\,ae_2+e_2\,d+ad+2\,be_2+ba+bc+ce_2+dc +db+ca \right) x^{2}\\
&-ba \left(e_2+d \right) \left( e_2+c \right),   \\
f_4(x,y)&=-2x^2y^2+ \left( d+c+2\,ce_2+ca+bc+2\,dc+
db+2\,e_2\,d+ad \right) x^{2}\\
&+ \Big( 2\,ae_2+ca+ad +a+2\,be_2 +2\,ba+bc+db+b \Big) y^{2} \\
&- \left( c+d \right) \left( a+b \right) e_2^{2}+ \left( -2\,db a-2\,adc-ad-2\,dbc-ca-db \right. \\
& \left. -bc-2\,bac \right) e_2-2\,dbac-adc-dba-bac-dbc,   
\end{align*}
\begin{align*}
f_3(x,y)&=(c+d)x^4-ba \left( 2\,e_2+d+c+1 \right)y^2+ (1+2\,a+2\,e_2+c+d+2\,b)x^2y^2\\
&+ba \Big( \left( c+d \right)e_2^{2}+ \left( d+c+2\,dc \right) e_2+dc \Big) \\
&+\Big( \left( -c-d \right) e_2^{2}+ \left( -2\,ad-2\,bc-2\,ca-2\,db-c-d-2\,dc \right) e_2\\
&-db-ca-bc-2\,adc-dc-dba-bac-ad-2\,dbc\Big)x^2, \\
f_2(x,y)&=(a+b)y^4-dc \left( 1+a+b+2\,e_2 \right)x^2+(a+b+2\,e_2+2\,c+2\,d+1)x^2y^2\\
&+dc \Big( \left( a+b \right) e_2^{2}+ \left( 2\,ba+a+b \right) e_2+ba \Big)\\
&+\Big(\left( -a-b \right) e_2^{2}+ \left( -a-2\,ba-2\,ad-2\,bc-b-2\,ca-2\,db \right) e_2\\
&-dbc-ba-ca-ad-2\,dba-bc-adc-2\,bac-db \Big)y^2, \\
f_1(x,y)&=x^4y^2+x^2y^4-cdx^4-aby^4+\Big( \left( -2\,c-2\,b-2\,d-1-2\,a \right) e_2\\
&-e_2^{2}-a-b-d-c -dc-ba-2\,ca -2\,db-2\,bc-2\,ad\Big)x^2y^2\\
&+dc \left( e_2^{2}+ \left( 2\,b+2\,a+1 \right) e_2+b+ba+a \right)x^2\\
&+ba \left( e_2^{2}+ \left( 2\,c+2\,d+1 \right) e_2+c+d+dc \right)y^2-adbe_2\,c \left( 1+e_2 \right).
\end{align*}
\end{proposition}
It has been shown \cite{njionou_et_al_2015} that the operators $\mathbf{D}_x$ and $\mathbf{S}_x$ satisfy the  following properties
\begin{gather*}
\mathbf{D}_x(fg)=\mathbf{D}_xf\mathbf{S}_xg+\mathbf{S}_xf\mathbf{D}_xg, \quad \,\,\,
\mathbf{S}_x(fg)=-x^2\mathbf{D}_xf\mathbf{D}_xg+\mathbf{S}_xf\mathbf{S}_xg, \\
\mathbf{D}_x\mathbf{S}_x=\mathbf{S}_x\mathbf{D}_x-\frac{1}{2}\mathbf{D}_x^2,  \qquad
\mathbf{S}_x^2=-x^2\mathbf{D}_x^2-\frac{1}{2}\mathbf{S}_x\mathbf{D}_x+\mathbf{I},
\end{gather*}
where $\mathbf{I}f=f$.

Applying the operators $\mathbf{D}_x$ and $\mathbf{D}_y$ on the divided-difference equation \eqref{eqdiv-diff-biv-wil} and using the above properties it yields:
\begin{proposition}\label{proposition:12}
The polynomial $W^{(1,0)}_{n,m}(x,y):=\mathbf{D}_xW_{n,m}(x,y)$ is solution of the following fourth-order linear partial divided-difference equation
\begin{multline*}
f_{11}(x,y) \mathbf{D}^2_x\mathbf{D}^2_y W^{(1,0)}_{n,m}(x,y)
+f_{21}(x,y)\mathbf{S}_x\mathbf{D}_x\mathbf{D}^2_y W^{(1,0)}_{n,m}(x,y)
+f_{31}(x,y)\mathbf{S}_y\mathbf{D}_y\mathbf{D}^2_x W^{(1,0)}_{n,m}(x,y)\\+f_{41}(x,y)
\mathbf{S}_x\mathbf{D}_x\mathbf{S}_y\mathbf{D}_y W^{(1,0)}_{n,m}(x,y)
+ f_{51}(x)\mathbf{D}^2_x W^{(1,0)}_{n,m}(x,y)+f_{61}(y)\mathbf{D}^2_y
R^{(1,0)}(x,y)\\
+f_{71}(x) \mathbf{S}_x\mathbf{D}_x W^{(1,0)}_{n,m}(x,y) +f_{81}(y)\mathbf{S}_y\mathbf{D}_y W^{(1,0)}_{n,m}(x,y)
+(m+n-1) (a+b+c+d+2e_2+m+n) W^{(1,0)}_{n,m}(x,y)=0,
\end{multline*}
where the coefficients $f_{i1},\ i=1,\ldots,8$ are given by
\begin{align*}
f_{81}(y)&=f_8(y)+\mathbf{D}_x(f_4(x,y)),\\
f_{71}(x)&=\mathbf{S}_x(f_7(x))-\frac{1}{2}\mathbf{D}_x(f_7(x))+\mathbf{D}_x(f_5(x)), \\
f_{61}(y)&=f_6(y)+\mathbf{D}_x(f_2(x,y)),\\
f_{51}(x)&=\mathbf{S}_x(f_5(x))-x^2 \mathbf{D}_x(f_7(x))-\frac{1}{2}\mathbf{S}_x(f_7(x)), \\
f_{41}(x,y)&=-\frac{1}{2}\mathbf{D}_x(f_4(x,y))+\mathbf{D}_x(f_3(x,y))+ \mathbf{S}_x(f_4(x,y)), 
\end{align*}
\begin{align*}
f_{31}(x,y)&=-\frac{1}{2}\mathbf{S}_x(f_4(x,y))+\mathbf{S}_x(f_3(x,y))-x^2 \mathbf{D}_x(f_4(x,y)),  \\
f_{21}(x,y)&=-\frac{1}{2}\mathbf{D}_x(f_2(x,y))+\mathbf{D}_x(f_1(x,y))+ \mathbf{S}_x(f_2(x,y)),\\
f_{11}(x,y)&=-\frac{1}{2}\mathbf{S}_x(f_2(x,y))+\mathbf{S}_x(f_1(x,y))-x^2\mathbf{D}_x(f_2(x,y))),
\end{align*}
and the polynomial $W^{(0,1)}_{n,m}(x,y):=\mathbf{D}_yW_{n,m}(x,y)$ is solution of the following fourth-order linear partial divided-difference equation
\begin{multline*}
f_{12}(x,y) \mathbf{D}^2_x\mathbf{D}^2_y W^{(0,1)}_{n,m}(x,y)
+ f_{22}(x,y)\mathbf{S}_x\mathbf{D}_x\mathbf{D}^2_y W^{(0,1)}_{n,m}(x,y)
+f_{32}(x,y)\mathbf{S}_y\mathbf{D}_y\mathbf{D}^2_x
W^{(0,1)}_{n,m}(x,y)\\
+f_{42}(x,y)
\mathbf{S}_x\mathbf{D}_x\mathbf{S}_y\mathbf{D}_y W^{(0,1)}_{n,m}(x,y)
+ f_{52}(x)\mathbf{D}^2_x W^{(0,1)}_{n,m}(x,y)+f_{62}(y)\mathbf{D}^2_y W^{(0,1)}_{n,m}(x,y)\\
+f_{72}(x) \mathbf{S}_x\mathbf{D}_x W^{(0,1)}_{n,m}(x,y) +f_{82}(y)\mathbf{S}_y\mathbf{D}_y W^{(0,1)}_{n,m}(x,y)
+(m+n-1) (a+b+c+d+2e_2+m+n) W^{(0,1)}_{n,m}(x,y)=0,
\end{multline*}
where the coefficients $f_{i2},\ i=1,\ldots,8$ are given by
\begin{align*}
f_{82}(y)&= \mathbf{S}_y(f_8(y))-\frac{1}{2}\mathbf{D}_y(f_8(y))+\mathbf{D}_y(f_6(y)),\\
f_{72}(x)&=f_7(x)+\mathbf{D}_y(f_4(x,y)),\\
f_{62}(y)&= -y^2\mathbf{D}_y(f_8(y))-\frac{1}{2}\mathbf{S}_y(f_8(y))+\mathbf{S}_y(f_6(y)),\\
f_{52}(x)&=f_5(x)+\mathbf{D}_y(f_3(x,y)),\\
f_{42}(x,y)&=-\frac{1}{2}\mathbf{D}_y(f_4(x,y))+\mathbf{D}_y(f_2(x,y))+ \mathbf{S}_y(f_4(x,y)),\\
f_{32}(x,y)&=-\frac{1}{2}\mathbf{D}_y(f_3(x,y))+\mathbf{D}_y(f_1(x,y))+ \mathbf{S}_y(f_3(x,y)), \\
f_{22}(x,y)&=-\frac{1}{2}\mathbf{S}_y(f_4(x,y))+\mathbf{S}_y(f_2(x,y))-y^2 \mathbf{D}_y(f_4(x,y)),\\
f_{12}(x,y)&=-\frac{1}{2}\mathbf{S}_y(f_3(x,y))+\mathbf{S}_y(f_1(x,y))-y^2\mathbf{D}_y(f_3(x,y)).
\end{align*}
\end{proposition}
We can therefore deduce that
\begin{corollary}
The polynomial $W^{(1,1)}_{n,m}(x,y):=\mathbf{D}_x\mathbf{D}_yW_{n,m}(x,y)$ is solution of the following fourth-order linear partial divided-difference
equation
\begin{multline*}
f_{13}(x,y) \mathbf{D}^2_x\mathbf{D}^2_y W^{(1,1)}_{n,m}(x,y)
+f_{23}(x,y)\mathbf{S}_x\mathbf{D}_x\mathbf{D}^2_y W^{(1,1)}_{n,m}(x,y)
+f_{33}(x,y)\mathbf{S}_y\mathbf{D}_y\mathbf{D}^2_x W^{(1,1)}_{n,m}(x,y)\\+f_{43}(x,y)
\mathbf{S}_x\mathbf{D}_x\mathbf{S}_y\mathbf{D}_y W^{(1,1)}_{n,m}(x,y)
+ f_{53}(x)\mathbf{D}^2_x W^{(1,1)}_{n,m}(x,y)+f_{63}(y)\mathbf{D}^2_y
W^{(1,1)}_{n,m}(x,y)+f_{73}(x) \mathbf{S}_x\mathbf{D}_x W^{(1,1)}_{n,m}(x,y)\\
+f_{83}(y)\mathbf{S}_y\mathbf{D}_y W^{(1,1)}_{n,m}(x,y)
+\left( m+n-2 \right) \left( m+a+b+c+d+2\,e_{{2}}+n+1 \right) W^{(1,1)}_{n,m}(x,y)=0,
\end{multline*}
where the coefficients $f_{i3},\ i=1,\ldots,8$ are given by
\begin{align*}
f_{83}(y)&=f_{82}(y)+\mathbf{D}_x(f_{42}(x,y)),\\
f_{73}(x)&=\mathbf{S}_x(f_{72}(x))-\frac{1}{2}\mathbf{D}_x(f_{72}(x))+\mathbf{D}_x(f_{52}(x)),
\end{align*}
\begin{align*}
f_{63}(y)&=f_{63}(y)+\mathbf{D}_x(f_{22}(x,y)), \\
f_{53}(x)&=\mathbf{S}_x(f_{52}(x))-x^2 \mathbf{D}_x(f_{72}(x))-\frac{1}{2}\mathbf{S}_x(f_{72}(x)),\\
f_{43}(x,y)&=-\frac{1}{2}\mathbf{D}_x(f_{42}(x,y))+\mathbf{D}_x(f_{32}(x,y))+ \mathbf{S}_x(f_{42}(x,y)), \\
f_{33}(x,y)&=-\frac{1}{2}\mathbf{S}_x(f_{42}(x,y))+\mathbf{S}_x(f_{32}(x,y))-x^2 \mathbf{D}_x(f_{42}(x,y)), \\
f_{23}(x,y)&=-\frac{1}{2}\mathbf{D}_x(f_{22}(x,y))+\mathbf{D}_x(f_{12}(x,y))+ \mathbf{S}_x(f_{22}(x,y)),\\
f_{13}(x,y)&=-\frac{1}{2}\mathbf{S}_x(f_{22}(x,y))+\mathbf{S}_x(f_{12}(x,y))-x^2\mathbf{D}_x(f_{22}(x,y))).
\end{align*}
\end{corollary}
\begin{remark}
\begin{enumerate}
\item It should be noted that as in the univariate Wilson case \cite[(9.1.7)]{MR2656096}
\begin{multline*}
\mathbf{D}_xw_n(x^2;a,b,c,d)=-n(n+a+b+c+d-1)w_{n-1}\left(x^2;a+\frac{1}{2},b+\frac{1}{2},c+\frac{1}{2},d+\frac{1}{2}\right),
\end{multline*}
the similar relation in the bivariate case is given by
\begin{equation*}
\mathbf{D}_xW_{n,m}(x,y;a,b,c,d;e_2)=-n(n+a+b+2e_2-1)W_{n-1,m}\left(x,y;a+\frac{1}{2},b+\frac{1}{2},c,d;e_2+\frac{1}{2}\right).
\end{equation*}
\item We would like also to emphasize that
\[f_{i1}=f_i\left(a+\frac{1}{2},b+\frac{1}{2},c,d,e_2+\frac{1}{2}\right),\quad i=1,\ldots,8,\] and
\[f_{i2}=f_i\left(a,b,c+\frac{1}{2},d+\frac{1}{2},e_2+\frac{1}{2}\right),\quad i=1,\ldots,8,\]
where $f_i=f_i(a,b,c,d,e_2),\, i=1,\ldots,8$ are given in Proposition \ref{prop:divided_biv_Wilson}, and $f_{i1}$ and $f_{i2}$ are given in Proposition \ref{proposition:12}.
\item If we consider the second family of bivariate Wilson polynomials \cite[Equation (2.13)]{MR1123596}
\begin{equation}\label{eq:Wilsonbar}
\bar{W}_{n,m}(x,y;a,b,c,d;e_2)=w_n(x^2;m+c+e_2,m+d+e_2,a,b)w_m(y^2;c,d,e_2+ix,e_2-ix),
\end{equation}
it follows that
\begin{equation*}
\mathbf{D}_y\bar{W}_{n,m}(x,y;a,b,c,d;e_2)=-m(m+c+d+2e_2-1) \bar{W}_{n,m-1}\left(x,y;a,b,c+\frac{1}{2},d+\frac{1}{2};e_2+\frac{1}{2}\right).
\end{equation*}

\item We also note that $W_{n,m}(x,y;a,b,c,d;e_2)$ and $\bar{W}_{n,m}(x,y;a,b,c,d;e_2)$ are solution of the same fourth-order divided-difference equation \eqref{eqdiv-diff-biv-wil}.
\end{enumerate}
\end{remark}
Using the change of variables \eqref{eq:ractowil}, we also show that
\begin{proposition}
Both families of bivariate Wilson polynomials~$W_{n,m}(x,y):=W_{n,m}(x,y;a,b,c,d;e_2)$ and~$\bar{W}_{n,m}(x,y):=\bar{W}_{n,m}(x,y;a,b,c,d;e_2)$~are respectively solution of the second-order divided-difference equations
\begin{multline*}
\Big({x}^{4}-{x}^{2}{y}^{2}+ \left( -2\,a{ e_2}-ba-2\,b{e_2}-{{e_2}}^{2} \right) {x}^{2}+ab{y}^{2}+abe_{2}^{2}\Big)\mathbf{D}_x^2W_{n,m}(x,y)+\Big(\left( a+2\,{e_2}+b \right) {x}^{2}\\
- \left( a+b \right) {y}^{2}-2\,ba{e_2}-be_{2}^{2}-ae_{2}^{2} \Big)\mathbf{S}_x\mathbf{D}_xW_{n,m}(x,y)-n(n-1+a+b+2e_2)W_{n,m}(x,y)=0,
\end{multline*}
and
\begin{multline*}
\Big(-{x}^{2}{y}^{2}+{y}^{4}+c{x}^{2}d+ \left( -2\,c{e_2}-dc-2\,d{e_2}-e_{2}^{2} \right) {y}^{2}+cde_{2}^{2}\Big)\mathbf{D}_y^2\bar{W}_{n,m}(x,y)+\Big( \left( -c-d \right) {x}^{2}\\
+ \left( c+2\,{e_2}+d \right) {y}^{2}-d e_{2}^{2}-2\,dc{e_2}-ce_{2}^{2}  \Big) \mathbf{S}_y \mathbf{D}_y\bar{W}_{n,m}(x,y)-m(m-1+c+d+2e_2)\bar{W}_{n,m}(x,y)=0.
\end{multline*}
\end{proposition}

From the divided-difference equation of the bivariate Wilson polynomials, we can also derive a difference equation they satisfy with rational coefficients as given in
\begin{proposition}
The bivariate Wilson polynomials $W_{n,m}(x,y):=W_{n,m}(x,y;a,b,c,d;e_2)$ are solution of the difference equation
\begin{multline}\label{eq:diffeqWilson}
F_1 W_{n,m}(x+i,y+i)+F_2 W_{n,m}(x+i,y-i)+F_3 W_{n,m}(x-i,y+i)
+F_4 W_{n,m}(x-i,y-i)+F_5 W_{n,m}(x+i,y)\\+F_6 W_{n,m}(x,y+i)+F_7 W_{n,m}(x-i,y)
+F_8 W_{n,m}(x,y-i)+F_9 W_{n,m}(x,y)=0,
\end{multline}
with
\[
F_1=\frac{f_1-xyf_4+i(xf_2+yf_3)}{4x(2x+i)y(2y+i)},\quad F_2=-\frac{f_1+xyf_4+i(xf_2-yf_3)}{4x(2x+i)y(-2y+i)},
\]
\[
F_3=\frac{-f_1-xyf_4+i(xf_2-yf_3)}{4x(-2x+i)y(2y+i)},\quad F_4=-{\frac {-f_{{1}}+if_{{2}}x+if_{{3}}y+f_{{4}}yx}{4 \left( -2\,y+i \right) y \left( -2\,x+i \right) x}},
\]
\[
F_5={\frac {-i \left( if_{{3}}-2\,f_{{2}}x+2\,if_{{1}}-4\,f_{{7}}x{y}
^{2}-f_{{7}}x-f_{{4}}x+4\,if_{{5}}{y}^{2}+if_{{5}} \right) }{2 \left( 2
\,x+i \right) x \left( 2\,y+i \right) \left( -2\,y+i \right) }},
 \]
\[
F_6={\frac {-\,i \left( 4\,if_{{6}}{x}^{2}+if_{{6}}+if_{{2}}-4\,f_{{8}}
y{x}^{2}-f_{{8}}y-f_{{4}}y-2\,f_{{3}}y+2\,if_{{1}} \right) }{ 2\left( 2
\,y+i \right) y \left( 2\,x+i \right) \left( -2\,x+i \right) }},
\]
\[
F_7={\frac {i \left( 2\,if_{{1}}+f_{{4}}x+if_{{3}}+4\,if_{{5}}{y}^{2}
+if_{{5}}+2\,f_{{2}}x+4\,f_{{7}}x{y}^{2}+f_{{7}}x \right) }{ 2\left( -2
\,x+i \right) x \left( 2\,y+i \right) \left( -2\,y+i \right) }},
\]
\[
F_8={\frac {i \left( 2\,if_{{1}}+4\,if_{{6}}{x}^{2}+if_{{6}}+f_{{4}}y
+if_{{2}}+2\,f_{{3}}y+4\,f_{{8}}y{x}^{2}+f_{{8}}y \right) }{2 \left( -2
\,y+i \right) y \left( 2\,x+i \right) \left( -2\,x+i \right) }},
\]
\begin{multline*}
F_9=(m + n)(2e_2 + a + b + c + d + m + n- 1)+ \\
{\frac {4 f_{{1}}+f_{{8}} (4{x}^{2}+1)+f_{{7}} (4{y}^{2}+1)
+f_{{6}}(8{x}^{2}+2)+f_{{4}}+2 (f_{{2}}+f_{{3}})+f_{{5}}(8
{y}^{2}+2)}{ \left( 2\,y+i \right) \left( -2\,y+i \right)
 \left( 2\,x+i \right) \left( -2\,x+i \right) }},
\end{multline*}
where $f_i,\, i=1,\ldots,8$ are the coefficients of the divided-difference equation \eqref{eqdiv-diff-biv-wil}.
\end{proposition}
\begin{proof}
Expand the expressions $\mathbf{D}^2_x\mathbf{D}^2_y W_{n,m}(x,y)$, $\mathbf{S}_x\mathbf{D}_x\mathbf{D}^2_yW_{n,m}(x,y)$, \ldots, $\mathbf{S}_y\mathbf{D}_y W_{n,m}(x,y)$ appearing in the divided-difference equation of the bivariate Wilson polynomials and collect with respect to $W_{n,m}(x+i,y+i)$, $W_{n,m}(x+i,y)$, \ldots, $W_{n,m}(x,y)$ to get the result.
\end{proof}

\subsection{Conjecture on the partial difference equation satisfied by the $p$-variate Wilson polynomials}\label{sec:conjwilson}

Let $p$ be a positive integer, ${\pmb{n}}=(n_{1},n_{2},\dots,n_{p})$, ${\pmb{x}}=(x_{1},x_{2},\dots,x_{p})$, and ${\pmb{e}}=(e_{2},\dots,e_{p})$. The multivariable Wilson polynomials of $p$ variables are defined by (see \cite{MR1123596})
\begin{multline}\label{eq:mvwilson}
W_{\pmb{n}}({\pmb{x}}; a,b,c,d;{\pmb{e}}_{p}) \\
=\Big(\prod_{k=1}^{p-1}w_{n_k}(x_k^2;N_1^{k-1}+a+E_2^k,N_1^{k-1}+b+E_2^k,e_{k+1}+ix_{k+1},e_{k+1}-ix_{k+1})\Big)\\
\times w_{n_p}(x_p^2;N_1^{p-1}+a+E_2^p,N_1^{p-1}+b+E_2^p,c,d),
\end{multline}
where $N_1^j=n_1+n_2+\cdots+n_j$, $E_2^j=e_2+e_3+\cdots+e_j$ with $N_1^0=E_2^1=0$. They are polynomials of total degree $N_1^p$ in the variable $x_1^2,x_2^2,\ldots,x_p^2$. In a similar way as in the case of the multivariate Racah polynomials ---see section \ref{sec:conjracah}---, for any $l_1,l_2,\ldots,l_p\in\{0,1,2\}$, we define the operator $E_{(l_1,l_2,\ldots,l_p)}$ which is equal to the product of $\mathbf{S}_{x_i}\mathbf{D}_{x_i}$ and $\mathbf{D}_{x_i}^2$ such that for $i=1,2,\ldots,p$ if $l_i=0$, there is not $\mathbf{S}_{x_i}\mathbf{D}_{x_i}$ and $\mathbf{D}_{x_i}^2$ in the product, if $l_i=1$, then there is $\mathbf{S}_{x_i}\mathbf{D}_{x_i}$ but not $\mathbf{D}_{x_i}^2$ in the product and if $l_i=2$, then there is $\mathbf{D}_{x_i}^2$ but not $\mathbf{S}_{x_i}\mathbf{D}_{x_i}$ in the product.
This operator is defined explicitly by
 \[E_{(l_1,l_2,\ldots,l_p)}=\prod_{i=0}^p\Big(\mathbf{S}_{x_i}\mathbf{D}_{x_i}\Big)^{-l_i(l_i-2)}\Big(\mathbf{D}_{x_i}^2\Big)^{\frac{1}{2}l_i(l_i-1)}.\]
Using these notations, we can write for $p=2$ Equation \eqref{eqdiv-diff-biv-wil} as
\begin{multline*}
f_1(x,y) E_{(2,2)}W_{n,m}(x,y) + f_2(x,y)E_{(1,2)}W_{n,m}(x,y)+f_3(x,y)E_{(2,1)}W_{n,m}(x,y)
+f_4(x,y) E_{(1,1)} W_{n,m}(x,y)\\
+ f_5(x)E_{(2,0)} W_{n,m}(x,y)
+f_6(y)E_{(0,2)} W_{n,m}(x,y)+f_7(x) E_{(1,0)}W_{n,m}(x,y)\\
+f_8(y)E_{(0,1)} W_{n,m}(x,y)+(m+n)(2e_2+a+b+c+d+m+n-1) W_{n,m}(x,y)=0.
\end{multline*}
We have the following conjecture
\begin{conjecture}\label{prop:multiWilson}
The $p$-variate Wilson polynomials $W_{\pmb{n}}({\pmb{x}}; a,b,c,d;{\pmb{e}}_{p})$ defined in \eqref{eq:mvwilson}
are solution of a $2p$-order partial linear divided-difference equation with polynomial coefficients ($f_i(x)$) of the form
\begin{multline*}
\sum_{\underset{l_1+l_2+\cdots+l_p=i}{i=1}}^{2p}f_i(\pmb{x})E_{(l_1,l_2,\ldots,l_p)}W_{\pmb{n}}({\pmb{x}}; a,b,c,d;{\pmb{e}}_{p}) \\
+(n_1+n_2+\cdots+n_p) (n_1+n_2+\cdots+n_p-1+a+b+c+d+2(e_2+e_3+\cdots+e_p))W_{\pmb{n}}({\pmb{x}}; a,b,c,d;{\pmb{e}}_{p}) =0,
\end{multline*}
where for any $i$ from 1 to $2p$, we take all the combinations of $l_1,l_2,\ldots,l_p\in\{0,1,2\}$ such that $l_1+l_2+\cdots+l_p=i$, $f_i(x)$ is a polynomial of degree $l_1+l_2+\cdots+l_p$ in the lattices $x_1^2,x_2^2,\ldots,x_p^2$ and if $l_j=0$, then $f_i$ does not depend on $x_j$.
 \end{conjecture}

\subsection{Coefficients of the three-term recurrence relations satisfied by bivariate Wilson polynomials solution of \eqref{eqdiv-diff-biv-wil} and connection between them}
In what follows we shall use the following bases
\begin{equation*}
F_n(x(s))=(-4)^{-n} \left(-2 s+\frac{1}{2}\right)_n \left(+2 s+\frac{1}{2}\right)_n,  \quad
F_n(y(t))=(-4)^{-n} \left(-2 t+\frac{1}{2}\right)_n \left(+2 t+\frac{1}{2}\right)_n ,
\end{equation*}
obtained from \eqref{eq:baseff1} and \eqref{eq:baseff2} by simply considering $\beta_{1}=\beta_{2}=0$, as well as
\begin{equation*}
\textbf{F}_n= (F_{n-k}(x(s))F_{k}(y(t)))\,,\quad 0 \leq k \leq n, \quad n\in \mathbb{N}_0\,,
\end{equation*}
and
\begin{equation*}
\textbf{x}^n= (s^{2(n-k)}t^{2k})\,,\quad 0 \leq k \leq n, \quad
n\in \mathbb{N}_0\,.
\end{equation*}
Therefore, a similar relation as \eqref{4} is satisfied where the coefficients are obtained from those in \eqref{eq:hhhh} by setting $\beta_{1}=\beta_{2}=0$.

The column vectors of bivariate Wilson polynomials satisfy a three-term recurrence relation of the form \eqref{RRTT}. The coefficients of the matrices can be obtained by using the same procedure as described for bivariate Racah polynomials, by using \eqref{eq:matricesgp} with $\lambda_{n}=n(2e_{2}+a+b+c+d+n-1)$.

\begin{proposition}
The matrices ${\mathbf{S}}_{n}$ of size $(n+1)\times n$ and $\mathbf{T}_n$ of size $(n+1)\times (n-1)$ have the same structure as in \eqref{eq:sn} and \eqref{eq:tn} respectively, and their coefficients are given in terms of the polynomial coefficients of the equation \eqref{eqdiv-diff-biv-wil} as
\begin{align*}
s_{k,k}&=\frac{1}{6} (-k+n+1) \left(6 (k-1) (n-k) \left(2 a+2 b+c+d+2 e_2+1\right)+(4 k-4 n+1) (k-n) \right. \\
& \left. \times \left(a+b+c+d+2 e_2\right)+6 (n-k) \left(e_2 \left(2 a+2 b+c+d+e_2\right)+a (b+c+d) \right. \right. \\
& \left. \left. +b (c+d)+c d\right) +6 (k-1) \left(a (2 b+c+d+1)+2 e_2 (a+b)+b (c+d+1)\right) \right. \\
&\left. +6 \left(e_2 \left(a (2 b+c+d)+e_2 (a+b)+b (c+d)\right) +a d (b+c)+a b c+b c d\right) \right. \\
&\left. +6 (k-2) (k-1) (a+b)+2 (n-k) (k (2 n-5)+(n-5) n+7)\right), \phantom{aaaaaaaaaaa} k=1,\dots,n, \\
s_{k+1,k}&=\frac{1}{6} k \left(2 e_2 \left(3 a (c+d+k-1)+3 b (c+d+k-1)+3 e_2 (c+d+k-1)+6 n (c+d+k-1) \right. \right. \\
& \left. \left. +6 c d-6 c-6 d-2 k^2-3 k+5\right)+a \left(6 b (c+d+k-1)+6 c (d+n-1)+6 n (d+k-1) \right. \right. \\
& \left. \left. -6 d-2 k^2-3 k+5\right) +b (6 c (d+n-1)+6 d (n-1)-(k-1) (2 k-6 n+5)) \right. \\
& \left. +6 n^2 (c+d+k-1)-2 n (-6 c (d-1)+6 d+(k-1) (2 k+5))-(k+1) (2 k (c+d-2) \right. \\
& \left. +6 c d-5 c-5 d+4)\right), \phantom{aaaaaaaaaaaaaaaaaaaaaaaaaaaaaaaaaaaaaaaaa} k=1,\dots,n,
\end{align*}
\begin{align*}
t_{k,k}&=\frac{1}{180} (n-k) (-k+n+1) \left(6 (k-n+1) \left(4 k^2+k (5-8 n)+n (4 n-5)-1\right) \left(a+b+c+d+2 e_2\right) \right. \\
&\left. -60 (k-1) (k-n) (k-n+1) \left(2a+2 b+c+d+2 e_2+1\right)+30 (k-1) (4 k-4 n+1) \left(a (2 b+c \right. \right. \\
& \left. \left. +d+1) +2 e_2 (a+b)+b (c+d+1)\right)-60 (k-n) (k-n+1) \left(e_2 \left(2 a+2 b+c+d+e_2\right) \right. \right. \\
& \left. \left.+a (b+c+d)+b (c+d)+c d\right)-30 (-4 k+4 n-1) \left(e_2 \left(a (2 b+c+d)+e_2 (a+b)+b (c+d)\right) \right. \right. 
\end{align*}
\begin{align*}
& \left. \left. +a d (b+c)+a b c+b c d\right)-180 a b (k-1) \left(c+d+2 e_2+1\right)-180 a b \left(c+e_2\right) \left(d+e_2\right) \right.  \\
& \left. +30 (k-2) (k-1) (a+b) (4 k-4 n+1)-180 a b (k-2) (k-1)-(k-n+1) \left(20 k^3+12 k^2 (n-14) \right. \right.  \\
& \left. \left. +k (205-24 (n-4) n)+n (-8 (n-9) n-193)-18\right)\right), \phantom{aaaaaaaaaaaaaaaaaa} k=1,\dots,n-1,  \\
t_{k+1,k}&=\frac{1}{18} k (n-k) \left(6 e_2 \left(-3 e_2 (c+d+k-1) (a+b-k+n-1)+a \left(-6 b (c+d+k-1) \right. \right. \right. \\
&\left. \left. \left. -6 n (c+d+k-1)-6 c d+9 c+9 d+2 k^2+6 k-8\right)+b \left(-6 n (c+d+k-1)-6 c d+9 c \right. \right. \right. \\
&\left. \left. \left.+9 d+2 k^2+6 k-8\right)+(k-n+1) (2 c (3 d+k+2 n-4)+2 d (k+2 n-4)+(k-1) (4 n-7))\right) \right. \\
& \left. +3 a \left(-2 b \left(3 k (c+d-2)+6 c d-6 c-6 d+k^2+5\right)-6 b n (c+d+k-1)-4 n^2 (c+d+k-1) \right. \right. \\
& \left. \left. +n (-4 k (c+d)+3 c (5-4 d)+15 d+13 (k-1))+6 c d k+12 c d+4 c k^2-10 c+4 d k^2-10 d \right. \right. \\
& \left. \left. +2 k^3-5 k^2-5 k+8\right)+3 b \left(c \left(6 d (k-2 n+2)+4 k^2-4 k n+(15-4 n) n-10\right) \right. \right.  \\
& \left. \left. +d \left(4 k^2-4 k n+(15-4 n) n-10\right)+(k-1) (k (2 k-3)+(13-4 n) n-8)\right) \right. \\
& \left. -(k-n+1) \left(3 c \left(2 d (k-4 n+5)-4 k n+k (2 k+3)-2 n^2+10 n-8\right) +3 d \left(-4 k n+k (2 k+3) \right. \right. \right. \\
& \left. \left. \left. -2 n^2+10 n-8\right)+(k-1) \left(4 k^2-4 k n-6 n^2+26 n-19\right)\right)\right), \phantom{aaaaaaaaaaaa} k=1,\dots,n-1,
\end{align*}
\begin{align*}
t_{k+2,k}&=\frac{1}{180} k (k+1) \left(180 c d (k-n+1) \left(a+b+2 e_2+1\right)+60 (k-1) k (k-n+1) \left(a+b+2 c+2 d \right. \right. \\
& \left. \left. +2 e_2+1\right)+30 (4 k-1) (k-n+1)\left(c (a+b+2 d+1)+d (a+b+1)+2 e_2 (c+d)\right) \right.  \\
& \left.  -6 (k-1) (k (4 k-5)-1) \left(a+b+c+d+2 e_2\right)-60 (k-1) k \left(e_2 \left(a+b+2 (c+d)+e_2\right) \right. \right. \\
& \left. \left. +a (b+c+d)+b (c+d)+c d\right)-30 (4 k-1) \left(e_2 \left(d (a+b+2 c)+c (a+b)+e_2 (c+d)\right) \right. \right. \\
& \left. \left. +a d (b+c)+a b c+b c d\right)-180 c d \left(a+e_2\right) \left(b+e_2\right)-180 c d (k-n+1) (k-n+2) \right. \\
& \left.-30 (4 k-1) (c+d) (k-n+1) (k-n+2)-(k-1) \left(20 k^3+24 k^2 (7-3 n) \right. \right. \\
& \left. \left. +5 k (12 (n-4) n+41)-12 n+18\right)\right), \phantom{aaaaaaaaaaaaaaaaaaaaaaaaaaaaaa} k=1,\dots,n-1.
\end{align*}
\end{proposition}

Moreover, the matrices of leading coefficients of bivariate Wilson polynomials defined in \eqref{eq:bwp} $W_{n,m}(x,y;a,b,c,d;e_2)$  are given by
\begin{equation}\label{eq:gnnbr3}
G_{n,n}=G_{n,n}(a,b,c,d;e_{2})=
\begin{pmatrix}
g_{r,s}(n,a,b,c,d;e_{2})
\end{pmatrix}_{0 \leq r,s \leq n},
\end{equation}
where
\begin{align*}
g_{r,s}(n,a,b,c,d;e_{2})=\begin{cases}
\displaystyle{0}, & r<s, \\[1mm]
\displaystyle{(-1)^{n-i-j}\,\binom{n-r}{s-r}\,(2 n - r - 1 + a + b + c + d + 2 e_{2})_{r}}\\
\qquad \times (a + b + n - s)_{r} \,(a + b + 2 e_{2} + n - r - 1)_{n-j} & r \geq s,
\end{cases}
\end{align*}
and the matrices of leading coefficients of bivariate Wilson polynomials $\bar{W}_{n,m}(x,y;a,b,c,d;e_2)$ defined in \eqref{eq:Wilsonbar} are given by
\begin{equation}\label{eq:gnnbr4}
G_{n,n}=\bar{G}_{n,n}(a,b,c,d;e_{2})=
\begin{pmatrix}
\bar{g}_{r,s}(n,a,b,c,d;e_{2})
\end{pmatrix}_{0 \leq r,s \leq n},
\end{equation}
where
\begin{align*}
\bar{g}_{r,s}(n,a,b,c,d;e_{2})=\begin{cases}
\displaystyle{0}, & r>s, \\[1mm]
\displaystyle{(-1)^n \binom{r}{s} (-c-d-r+1)_{i-j} (c+d+r+2 e_{2}-1)_{s}}\\
\qquad \times \displaystyle{ (a+b+c+d+r+n+2 e_{2}-1)_{n-r}}, & r \leq s.
\end{cases}
\end{align*}
As a consequence, both families of bivariate Wilson polynomials defined in \eqref{eq:bwp} and \eqref{eq:Wilsonbar} can be generated from the three-term recurrence relations they satisfy by using Theorem \ref{eq:matricesTTRR}, where for these specific families we have $x_{1}(x)=x^{2}$ and $y_{2}(y)=y^{2}$. Moreover, if we consider $G_{n,n}$ as the identity matrix we can introduce the family of monic bivariate Wilson polynomials by using Corollary \ref{eq:matricesTTRRmonic}. Finally, by using the matrices ${G}_{n,n}(a,b,c,d;e_{2})$ and $\bar{G}_{n,n}(a,b,c,d;e_{2})$ defined above, it is possible to solve the connection problem between the two families of bivariate Wilson polynomials defined in \eqref{eq:bwp} and \eqref{eq:Wilsonbar}, in a similar way as described for bivariate Racah polynomials \eqref{eq:connection} in section \ref{section:mbr}.

\subsection{Fourth-order linear partial divided-difference equation of the bivariate continuous dual Hahn polynomials}
The continuous dual Hahn polynomials $d_n(a,b,c \vert x)$ result upon dividing \eqref{eq:wilson} by $d^n$ and taking the limit $d\to\infty$, (see \cite{MR1123596})
\begin{equation*}
d_n(a,b,c \vert x)=(a+b)_n(a+c)_n\hyper{3}{2}{-n,a+ix,a-ix}{a+b,a+c}{1}.
\end{equation*}
Dividing \eqref{eq:bwp} by $b^{n+m}$ and taking the limit $b\to \infty$ yields (after redefining $c\to b$, $d\to c$) the bivariate continuous dual Hahn polynomials \cite{MR1123596}
\begin{equation}\label{eq:vcdhp1}
D_{n,m}(a,b,c,e_2;x,y)=d_n(a,e_2+iy,e_2-iy \vert x)d_m(n+a+e_2,b,c \vert y ).
\end{equation}
Using the above limit process, we deduce from \eqref{eqdiv-diff-biv-wil} that

\begin{proposition}\label{prop:divided_bivCDH}
The bivariate continuous dual Hahn polynomials~$D_{n,m}(a,b,c,e_2;x,y)$ are solution of the fourth-order linear partial divided-difference equation
\begin{multline}\label{eqdiv-diff-biv-CDH}
f_1(x,y) \mathbf{D}^2_x\mathbf{D}^2_yD_{n,m}(x,y) + f_2(x,y)\mathbf{S}_x\mathbf{D}_x\mathbf{D}^2_yD_{n,m}(x,y)+f_3(x,y)\mathbf{S}_y\mathbf{D}_y\mathbf{D}^2_xD_{n,m}(x,y)\\
+f_4(x,y) \mathbf{S}_x\mathbf{D}_x\mathbf{S}_y\mathbf{D}_y D_{n,m}(x,y)
+ f_5(x)\mathbf{D}^2_x D_{n,m}(x,y)+f_6(y)\mathbf{D}^2_y D_{n,m}(x,y)+f_7(x) \mathbf{S}_x\mathbf{D}_xD_{n,m}(x,y) \\
+f_8(y)\mathbf{S}_y\mathbf{D}_y D_{n,m}(x,y)+(m+n) D_{n,m}(x,y)=0,
\end{multline}
where~$D_{n,m}(x,y):=D_{n,m}(a,b,c,e_2;x,y)$ and
\begin{align*}
f_8(y)&=-{y}^{2}+ \left( c+b \right) e_2+cb+ab+ca,\\
f_7(x)&=-{x}^{2}+e_2^{2}+ \left( c+2\,a+b \right) e_2+ca+c b +ab,\\
f_6(y)&=-cb \left( a+e_2 \right) + \left( c+b+e_2+a \right) {y}^{2},
\end{align*}
\begin{align*}
f_5(x)&=-a \left( e_2+c \right) \left( e_2+b \right) + \left( 2\,e_2+a+b+c \right) {x}^{2},\\
f_4(x,y)&=\left( -b-c \right) {e_2}^{2}+ \left( -2\,ca-2\,ab-b-2\,cb-c\right) {e_2}-ca-2\,bac\\ &-cb-ab + \left( 1+2\,e_2+2\,a+b+c
 \right) {y}^{2}+ \left( c+b \right) {x}^{2}, \\
f_3(x,y)&=a \left( e_2\,b+2\,be_2\,c+ce_2+cb+c{e_2}^{2}+b{e_2}^{2} \right) +2\,{x}^{2}{y}^{2}-a \left( 2\,e_2+c+1+b \right) {y}^{2}\\
&+ \left( -b-ab-2\,e_2\,b-2\,cb-c-2\,ce_2-ca \right) {x}^{2},\\
f_2(x,y)&=cb \left( 2\,ae_2+a+e_2+{e_2}^{2} \right) +{x}^{2}{y}^{2}-{x}^{2}cb+{y}^{4} + \left( -e_2-a-b \right. \\
& \left. -{e_2}^{2}-2\,e_2\,b-2\,ae_2-2 \,ca-c-2\,ce_2-cb-2\,ab \right) {y}^{2},\\
f_1(x,y)&=-be_2\,ca \left( 1+e_2 \right) + \left( -1-2\,c-2\,e_2-2\,b-a \right) {x}^{2}{y}^{2}+cb \left( 1+2\,e_2+a \right) {x}^{2}-a{y}^{4} \\
&+a \left( cb+e_2+b+2\,e_2\,b+{e_2}^{2}+2\,ce_2+c \right) {y}^{2}.
\end{align*}
\end{proposition}
The divided-difference equation \eqref{eqdiv-diff-biv-CDH} is equivalent to a difference equation of the form \eqref{eq:diffeqWilson} in which the $f_i$, $i=1,\ldots,8$ are those of Proposition \ref{prop:divided_bivCDH}.

The following divided-derivative of the continuous dual Hahn polynomials is valid:
\[\mathbf{D}_xD_{n,m}(a,b,c,e_2;x,y)=-nD_{n-1,m}(a+1/2,b,c,e_2+1/2;x,y).\]
We deduce from this relation that $D_{n,m}^{(1,0)}(x,y):=\mathbf{D}_xD_{n,m}(x,y)$ is solution of the fourth-order linear partial divided-difference equation
\begin{multline*}
f_{11}(x,y) \mathbf{D}^2_x\mathbf{D}^2_y D_{n,m}^{(1,0)}(x,y)
+ f_{21}(x,y)\mathbf{S}_x\mathbf{D}_x\mathbf{D}^2_y D_{n,m}^{(1,0)}(x,y)
+f_{31}(x,y)\mathbf{S}_y\mathbf{D}_y\mathbf{D}^2_x
D_{n,m}^{(1,0)}(x,y)\\
+f_{41}(x,y)\mathbf{S}_x\mathbf{D}_x\mathbf{S}_y\mathbf{D}_y D_{n,m}^{(1,0)}(x,y)
+ f_{51}(x)\mathbf{D}^2_x
D_{n,m}^{(1,0)}(x,y)+f_{61}(y)\mathbf{D}^2_y
D_{n,m}^{(1,0)}(x,y)+f_{71}(x) \mathbf{S}_x\mathbf{D}_x
D_{n,m}^{(1,0)}(x,y) \\
+f_{81}(y)\mathbf{S}_y\mathbf{D}_y D_{n,m}^{(1,0)}(x,y)+(m+n-1)D_{n,m}^{(1,0)}(x,y)=0,
\end{multline*}
where $f_{1i}=f_i(a+1/2,b,c,e_2+1/2),\, i=1,\ldots,8$ with $f_i=f_i(a,b,c,e_2)$ given in Proposition \ref{prop:divided_bivCDH}. It also follows that the continuous dual Hahn polynomials $D_{n,m}(a,b,c,e_2;x,y):=D_{n,m}(x,y)$ are solution of the second-order divided-difference equation
\[
\Big(\left( -a-2\,{e_2} \right) {x}^{2}+{y}^{2}a+ae_{2}^{2}\Big)\mathbf{D}^2_xD_{n,m}(x,y)+\Big({x}^{2}-{y}^{2}-2\,a{e_2}-e_{2}^{2}\Big)\mathbf{S}_x\mathbf{D}_x D_{n,m}(x,y)-nD_{n,m}(x,y)=0.
\]

\subsubsection{Coefficients of the three-term recurrence relations satisfied by bivariate continuous dual Hahn polynomials and the family of monic bivariate continuous dual Hahn polynomials}

The column vectors of bivariate continuous dual Hahn polynomials satisfy a three-term recurrence relation of the form \eqref{RRTT}. The coefficients of the matrices can be obtained by using the same procedure as described for bivariate Racah polynomials, by using \eqref{eq:matricesgp} with $\lambda_{n}=n$.
\begin{proposition}
The matrices ${\mathbf{S}}_{n}$ of size $(n+1)\times n$ and $\mathbf{T}_n$ of size $(n+1)\times (n-1)$ have the same structure as in \eqref{eq:sn} and \eqref{eq:tn} respectively, and their coefficients are given in terms of the polynomial coefficients of the equation \eqref{eqdiv-diff-biv-CDH} as
\begin{align*}
s_{k,k}&=-\frac{1}{6} (k-n-1) \left(6 e_2 \left(2 a+b+c+e_2+2 n-2\right)+6 a (b+c+k+n-2)+6 b (c+n-1) \right. \\
& \left. +n (6 c+4 k-13)-6 c-2 k^2+k+4 n^2+6\right), \phantom{aaaaaaaaaaaaaaaaaaaaaaaaa} k=1,\dots,n, 
\end{align*}
\begin{align*}
s_{k+1,k}&=\frac{1}{6} k \left(6 a (b+c+k-1)+6 e_2 (b+c+k-1)+6 b (c+n-1) \right. \\
& \left. +6 n (c+k-1)-6 c-2 k^2-3 k+5\right), \phantom{aaaaaaaaaaaaaaaaaaaaaaaaaaaaaa} k=1,\dots,n,  \\
t_{k,k}&=\frac{1}{30} (n-k) (n-k+1) \left(-10 (k-n) (k-n+1) \left(a+b+c+2 e_2\right)+5 (k-1) (4 k-4 n+1) \right. \\
& \left. \times \left(2a+b+c+2 e_2+1\right)-5 (-4 k+4 n-1) \left(e_2 \left(2 a+b+c+e_2\right)+a (b+c)+b c\right) \right. \\
& \left. -30 a (k-1) \left(b+c+2 e_2+1\right)-30 a \left(b+e_2\right) \left(c+e_2\right)-30 a (k-2) (k-1)+4 k^3+8 k^2 n \right. \\
& \left. -46 k^2-8 k n^2+22 k n+49 k-4 n^3+29 n^2-64 n+9\right), \phantom{aaaaaaaaaaaaaaaaa} k=1,\dots,n-1, \\
t_{k+1,k}&=-\frac{1}{6} k (k-n) \left(2 e_2 \left(-6 a (b+c+k-1)-3 e_2 (b+c+k-1)-6 n (b+c+k-1)-6 b c \right. \right. \\
& \left. \left. +9 b+9 c+2 k^2+6 k-8\right)-2 a \left(3 k (b+c-2)+6 b c-6 b-6 c+k^2+5\right)-6 a n (b+c+k-1) \right. \\
& \left.-4 n^2 (b+c+k-1)+n (-4 k (b+c)+3 b (5-4 c)+15 c+13 (k-1))+6 b c k+12 b c+4 b k^2 \right. \\
& \left. -10 b+4 c k^2-10 c+2 k^3-5 k^2-5 k+8\right), \phantom{aaaaaaaaaaaaaaaaaaaaaaaaaaa} k=1,\dots,n-1, \\
t_{k+2,k}&=\frac{1}{30} k (k+1) \left(-10 (k-1) k \left(a+b+c+e_2\right)-5 (4 k-1) \left(a (b+c)+e_2 (b+c)+b c\right) \right. \\
& \left. -30 b c \left(a+e_2\right)+30 b c (k-n+1)+5 (4 k-1) (b+c) (k-n+1) \right. \\
& \left. +(k-1) (k (6 k-10 n+15)+1)\right), \phantom{aaaaaaaaaaaaaaaaaaaaaaaaaaaaaaaa} k=1,\dots,n-1.
\end{align*}
\end{proposition}

The matrices of leading coefficients of continuous dual Hahn polynomials defined in \eqref{eq:vcdhp1} $D_{n,m}(a,b,c,e_2;x,y)$ are given by
\begin{equation}\label{eq:gnnbr7}
G_{n,n}=G_{n,n}(a,b,c,e_2)=
\begin{pmatrix}
g_{i,j}(a,b,c,e_2)
\end{pmatrix}_{0 \leq r,s \leq n},
\end{equation}
where
\begin{align*}
g_{r,s}(a,b,c,e_2)=\begin{cases}
\displaystyle{0}, & r > s, \\[3mm]
\displaystyle{(-1)^{n-r-s} \binom{n-r}{s-r}},& r \leq s.
\end{cases}
\end{align*}
As a consequence, bivariate continuous dual Hahn polynomials can be generated from the three-term recurrence relations they satisfy by using Theorem \ref{eq:matricesTTRR}, where for this specific family we have $x_{1}(x)=x^2$ and $y_{2}(y)=y^2$. Moreover, if we consider $G_{n,n}$ as the identity matrix we can introduce the family of monic bivariate continuous dual Hahn polynomials by using Corollary \ref{eq:matricesTTRRmonic}.

\subsection{Linear partial divided-difference equation of the bivariate continuous Hahn polynomials}
The continuous Hahn polynomials $h_n(a,b,c,d \vert x)$ are obtained by transforming \cite{MR1123596}
\[
a\to a+\frac{1}{2}i\epsilon,\quad
b\to b-\frac{1}{2}i\epsilon,\quad
c\to c+\frac{1}{2}i\epsilon,\quad
 d\to d-\frac{1}{2}i\epsilon,\quad
 x\to x-\frac{1}{2}\epsilon,
 \]
dividing \eqref{eq:wilson} by $\epsilon^n$ and then taking the limit $\epsilon\to \infty$. The resulting polynomials are
\begin{equation}\label{eq:unimissing}
h_n(a,b,c,d\vert x)=i^n(a+b)_n(a+d)_n\hyper{3}{2}{-n,n+a+b+c+d-1,a+ix}{a+b,a+d}{1}.
\end{equation}
Bivariate continuous Hahn polynomials can also be obtained from the bivariate Wilson families \eqref{eq:bwp} and \eqref{eq:Wilsonbar} by transforming \cite{MR1123596}
\begin{equation}\label{eq:WilsontoCH}
\begin{cases}
\displaystyle{a\to a_1+\frac{1}{2}i\epsilon,\quad , b\to b_1-\frac{1}{2}i\epsilon,\quad c\to b_3+\frac{1}{2}i\epsilon,} \\[3mm]
\displaystyle{d\to a_3-\frac{1}{2}i\epsilon,\quad x\to x-\frac{1}{2}\epsilon,\quad y\to y-\frac{1}{2}\epsilon,}
\end{cases}
\end{equation}
dividing \eqref{eq:bwp} and \eqref{eq:Wilsonbar} by $\epsilon^{n+m}$ and then taking the limit $\epsilon \to \infty$. This yields the families
\begin{equation}\label{eq:bchp1}
H_{n,m}(a_1,e_2,a_3;b_1,b_3;x,y)=h_n(a_1,b_1,e_2-iy,e_2+iy \vert x)h_m(n+a_1+e_2,n+b_1+e_2,b_3,a_3 \vert y),
\end{equation}
and
\begin{equation}\label{eq:bchp2}
\bar{H}_{n,m}(a_1,e_2,a_3;b_1,b_3;x,y)= h_n(m+e_2+b_3,m+e_2+a_3,a_1,b_1 \vert x) h_m(b_3,a_3,e_2-ix,e_2+ix \vert y).
\end{equation}
Whereas the Wilson operator $\mathbf{D}_x$ is appropriate for the Wilson and the bivariate Wilson polynomials, the corresponding operator for the continuous Hahn and the bivariate continuous Hahn polynomials is the operator \cite[p. 436]{NIST2010}
\[
\delta_x=\frac{f(x+{\frac{i}{2}})-f(x-{\frac{i}{2}})}{i}.
\]
The following limit relations between $\mathbf{D}_x$ and $\delta_x$ hold:
\begin{multline}\label{dx_Wilson_to_CH1}
\lim_{\epsilon\to\infty}\mathbf{D}_x\frac{W_{n,m}(x-\frac{1}{2}\epsilon,y-\frac{1}{2}\epsilon;a_1+\frac{1}{2}i\epsilon,b_1-\frac{1}{2}i\epsilon,b_3+\frac{1}{2}i\epsilon,a_3-\frac{1}{2}i\epsilon;e_2)}{\epsilon^{n+m-1}}
\\
=-\delta_x H_{n,m}(a_1,e_2,a_3;b_1,b_3;x,y),
\end{multline}
\begin{multline}
\lim_{\epsilon\to\infty}\mathbf{S}_x\frac{W_{n,m}(x-\frac{1}{2}\epsilon,y-\frac{1}{2}\epsilon;a_1+\frac{1}{2}i\epsilon,b_1-\frac{1}{2}i\epsilon,b_3+\frac{1}{2}i\epsilon,a_3-\frac{1}{2}i\epsilon;e_2)}{\epsilon^{n+m}} \\
=\mathbf{S}_x H_{n,m}(a_1,e_2,a_3;b_1,b_3;x,y),
\end{multline}
\begin{multline}
\lim_{\epsilon\to\infty}\mathbf{D}_y\frac{W_{n,m}(x-\frac{1}{2}\epsilon,y-\frac{1}{2}\epsilon;a_1+\frac{1}{2}i\epsilon,b_1-\frac{1}{2}i\epsilon,b_3+\frac{1}{2}i\epsilon,a_3-\frac{1}{2}i\epsilon;e_2)}{\epsilon^{n+m-1}} \\
=-\delta_y H_{n,m}(a_1,e_2,a_3;b_1,b_3;x,y),
\end{multline}
\begin{multline}\label{sy_Wilson_to_CH1}
\lim_{\epsilon\to\infty}\mathbf{S}_y\frac{W_{n,m}(x-\frac{1}{2}\epsilon,y-\frac{1}{2}\epsilon;a_1+\frac{1}{2}i\epsilon,b_1-\frac{1}{2}i\epsilon,b_3+\frac{1}{2}i\epsilon,a_3-\frac{1}{2}i\epsilon;e_2)}{\epsilon^{n+m}} \\
=\mathbf{S}_y H_{n,m}(a_1,e_2,a_3;b_1,b_3;x,y).
\end{multline}
Applying the transformations \eqref{eq:WilsontoCH} to  \eqref{eqdiv-diff-biv-wil} and using \eqref{dx_Wilson_to_CH1}--\eqref{sy_Wilson_to_CH1}, it follows that
\begin{proposition}\label{divided-difference_eq_CH}
The bivariate continuous Hahn polynomials $H_{n,m}(x,y):=H_{n,m}(a_1,e_2,a_3;b_1,b_3;x,y)$ and ${H}_{n,m}(x,y):=\bar{H}_{n,m}(a_1,e_2,a_3;b_1,b_3;x,y)$ are solution of the fourth-order linear partial divided-difference equation
\begin{multline}\label{eqdiv-diff-biv-CH}
f_1(x,y) \delta^2_x\delta^2_yH_{n,m}(x,y) + f_2(x,y)\mathbf{S}_x\delta_x\delta^2_yH_{n,m}(x,y)
+f_3(x,y)\mathbf{S}_y\delta_y\delta^2_xH_{n,m}(x,y)\\+f_4(x,y) \mathbf{S}_x\delta_x\mathbf{S}_y\delta_y H_{n,m}(x,y)
+ f_5(x)\delta^2_x H_{n,m}(x,y)+f_6(y)\delta^2_y H_{n,m}(x,y)+f_7(x) \mathbf{S}_x\delta_xH_{n,m}(x,y)\\
+f_8(y)\mathbf{S}_y\delta_y H_{n,m}(x,y)
+\left( n+m \right)  \left( a_{{1}}-1+2\,e_{{2}}+b_{{3}}+a_{{3}}+b_{{1}}+m+n \right) H_{n,m}(x,y)=0,
\end{multline}
where
\begin{align*}
f_8(y)&=i \left( a_{{1}}b_{{3}}+e_{{2}}b_{{3}}-e_{{2}}a_{{3}}-b_{{1}}a_{{3}} \right) + \left( -a_{{1}}-b_{{1}}-2\,e_{{2}}-b_{{3}}-a_{{3}} \right) y,\\
f_7(x)&=i \left( a_{{1}}b_{{3}}-b_{{1}}a_{{3}}-b_{{1}}e_{{2}}+a_{{1}}e_{{2}} \right) + \left( -a_{{1}}-b_{{1}}-2\,e_{{2}}-b_{{3}}-a_{{3}} \right) x, \\
f_6(y)&=\frac{1}{2}\,a_{{1}}b_{{3}}+\frac{1}{2}\,e_{{2}}b_{{3}}+\frac{1}{2}\,e_{{2}}a_{{3}}+\frac{1}{2}\,b_{{1}}a_{{3}}-{y}^{2}+\frac{1}{2}\,i \left( a_{{1}}+b_{{3}}-a_{{3}}-b_{{1}} \right) y, \\
f_5(x)&=\frac{1}{2}\,a_{{1}}b_{{3}}+\frac{1}{2}\,b_{{1}}a_{{3}}+\frac{1}{2}\,a_{{1}}e_{{2}}+\frac{1}{2}\,b_{{1}}e_{{2}}+\frac{1}{2}\,i \left( a_{{1}}+b_{{3}}-a_{{3}}-b_{{1}} \right) x-{x}^{2}, \\
f_4(x,y)&=a_{{1}}b_{{3}}+b_{{1}}a_{{3}}-2\,xy-i \left( -b_{{3}}+a_{{3}} \right)x+i \left( -b_{{1}}+a_{{1}} \right) y, \\
f_3(x,y)&=-\frac{1}{2}\,i \left( a_{{1}}b_{{3}}-b_{{1}}a_{{3}} \right) + \left( \frac{1}{2}\,b_{{3}}+\frac{1}{2}\,a_{{3}} \right) x+ \left( \frac{1}{2}\,a_{{1}}+\frac{1}{2}\,b_{{1}} \right)y, \\
f_2(x,y)&=-\frac{1}{2}\,i \left( a_{{1}}b_{{3}}-b_{{1}}a_{{3}} \right) + \left( \frac{1}{2}\,b_{{3}}+\frac{1}{2}\,a_{{3}} \right) x+ \left( \frac{1}{2}\,a_{{1}}+\frac{1}{2}\,b_{{1}} \right)y, \\
f_1(x,y)&=-1/4\,a_{{1}}b_{{3}}-1/4\,b_{{1}}a_{{3}}+\frac{1}{2}\,xy+1/4\,i \left( -b_{{3}}+a_{{3}} \right) x-1/4\,i \left( -b_{{1}}+a_{{1}} \right) y.
\end{align*}
\end{proposition}
Equation \eqref{eqdiv-diff-biv-CH} is equivalent to the difference equation
\begin{multline*}
F_1 H_{n,m}(x+i,y+i)+F_2 H_{n,m}(x+i,y-i)+F_3 H_{n,m}(x-i,y+i)\\
+F_4 H_{n,m}(x-i,y-i)+F_5 H_{n,m}(x+i,y)+F_6 H_{n,m}(x,y+i)\\
+F_7 H_{n,m}(x-i,y)+F_8 H_{n,m}(x,y-i)+F_9 H_{n,m}(x,y)=0,
\end{multline*}
with
\[F_1=f_1+{\frac{i}{2}}(f_2+f_3)-\frac{1}{4}f_4, \quad F_2=f_1+{\frac{i}{2}}(f_2-f_3)+\frac{1}{4}f_4,F_3=f_1-{\frac{i}{2}}(f_2-f_3)+\frac{1}{4}f_4,\]
\[ F_4=f_1-{\frac{i}{2}}(f_2+f_3)-\frac{1}{4}f_4, F_5=-2f_1-f_5-if_2-\frac{1}{2}if_7,\quad F_6=-2f_1-if_3-f_6-\frac{1}{2}if_8, \]
\[F_7=-2f_1+if_2-f_5+\frac{1}{2}if_7,\quad F_8=if_3-2f_1-f_6+\frac{1}{2}if_8, \]
\[F_9=4f_1+2f_6+2f_5+\left( n+m \right)  \left( a_{{1}}-1+2\,e_{{2}}+b_{{3}}+a_{{3}}+b_{{1
}}+m+n \right).\]
The following partial differences of the bivariate continuous Hahn polynomials are valid:
\begin{multline*}
\delta_xH_{n,m}(a_1,e_2,a_3;b_1,b_3;x,y)=n(n+a_1+b_1+2e_2-1)H_{n-1,m}\left(a_1+\frac{1}{2},e_2+\frac{1}{2},a_3;b_1+\frac{1}{2},b_3;x,y \right),
\end{multline*}
\begin{multline*}
\delta_y\bar{H}_{n,m}(a_1,e_2,a_3;b_1,b_3;x,y)\\=m(m+a_3+b_3+2e_2-1) \bar{H}_{n,m-1}\left(a_1,e_2+\frac{1}{2},a_3+\frac{1}{2};b_1,b_3+\frac{1}{2};x,y \right).
\end{multline*}
The immediate consequence of the above partial derivatives is
\begin{proposition}
  The partial difference-derivative of bivariate continuous Hahn polynomials
  \begin{align*}
H_{n,m}^{(1,0)}(x,y):=  \delta_xH_{n,m}(a_1,e_2,a_3;b_1,b_3;x,y), \\
H_{n,m}^{(0,1)}(x,y):=  \delta_yH_{n,m}(a_1,e_2,a_3;b_1,b_3;x,y),
  \end{align*}
 are respectively solution of the fourth-order linear partial divided-difference equations
\begin{multline*}
f_{11}(x,y) \delta^2_x\delta^2_y H_{n,m}^{(1,0)}(x,y)
+ f_{21}(x,y)\mathbf{S}_x\delta_x\delta^2_y H_{n,m}^{(1,0)}(x,y)
+f_{31}(x,y)\mathbf{S}_y\delta_y\delta^2_x
H_{n,m}^{(1,0)}(x,y)\\
+f_{41}(x,y)
\mathbf{S}_x\delta_x\mathbf{S}_y\delta_y H_{n,m}^{(1,0)}(x,y)
+ f_{51}(x)\delta^2_x
H_{n,m}^{(1,0)}(x,y)+f_{61}(y)\delta^2_y
H_{n,m}^{(1,0)}(x,y)+f_{71}(x) \mathbf{S}_x\delta_x H_{n,m}^{(1,0)}(x,y)\\
+f_{81}(y)\mathbf{S}_y\delta_y H_{n,m}^{(1,0)}(x,y)
+\left(m+n-1 \right)  \left( m+n+a_{{1}}+2\,e_{{2}}+b_{{3}}+a_{{3}}+b_{{1}} \right)H_{n,m}^{(1,0)}(x,y)=0
\end{multline*}
and
\begin{multline*}
f_{12}(x,y) \delta^2_x\delta^2_y H_{n,m}^{(0,1)}(x,y)
+ f_{22}(x,y)\mathbf{S}_x\delta_x\delta^2_y H_{n,m}^{(0,1)}(x,y)
+f_{32}(x,y)\mathbf{S}_y\delta_y\delta^2_x
H_{n,m}^{(0,1)}(x,y)\\
+f_{42}(x,y)
\mathbf{S}_x\delta_x\mathbf{S}_y\delta_y H_{n,m}^{(0,1)}(x,y)
+ f_{52}(x)\delta^2_x
H_{n,m}^{(0,1)}(x,y)+f_{62}(y)\delta^2_y
H_{n,m}^{(0,1)}(x,y)+f_{72}(x) \mathbf{S}_x\delta_x H_{n,m}^{(0,1)}(x,y)\\
+f_{82}(y)\mathbf{S}_y\delta_y H_{n,m}^{(0,1)}(x,y)
+\left(m+n-1 \right)  \left( m+n+a_{{1}}+2\,e_{{2}}+b_{{3}}+a_{{3}}+b_{{1}} \right) H_{n,m}^{(0,1)}(x,y)=0,
\end{multline*}
with  $f_{i1}=f_i(a_1+\frac{1}{2},e_2+\frac{1}{2},a_3,b_1+\frac{1}{2},b_3)$ and $f_{i2}=f_i(a_1,e_2+\frac{1}{2},a_3+\frac{1}{2},b_1,b_3+\frac{1}{2})$ for
$i=1,\ldots,8$ where the coefficients $f_i=f_i(a_1,e_2,a_3,b_1,e_3)$ are given in Proposition \ref{divided-difference_eq_CH}.
\end{proposition}

\subsection{Trivariate continuous Hahn polynomials}

To illustrate the truthful of our conjectures presented in sections \ref{sec:conjracah} and \ref{sec:conjwilson}, we consider the case $p=3$ for the trivariate continuous Hahn polynomials defined by \cite{MR1123596}
\begin{multline*}
H_{n,m,r}(a_1,e_2,e_3,a_4;b_1,b_4;x,y,z)=h_n(a_1,b_1,e_2-iy,e_2+iy \vert x)\\
\times h_m(n+a_1+e_2,n+b_1+e_2,e_3-iz,e_3+iz \vert y) \\ \times
h_{r}(n+m+a_1+e_2+e_3,n+m+b_1+e_2+e_3,b_4,a_4 \vert z),
\end{multline*}
where the continuous Hahn polynomials $h_{j}$ are defined in \eqref{eq:unimissing}. It follows that
\begin{proposition}
The trivariate continuous Hahn polynomials
\[
H_{n,m,r}(x,y,z):=H_{n,m,r}(a_1,e_2,e_3,a_4;b_1,b_4;x,y,z)
\]
are solution of the six-order partial linear divided-difference equation with $3^3=27$ polynomial coefficients
\begin{multline*}
\left( n+m+r \right)  \left( n+m+r-1+a_{{1}}+2e_{{2}}+2e_{{3}}+a_{{4}}+
b_{{1}}+b_{{4}} \right)
H_{n,m,r}(x,y,z)+f_1\mathbf{S}_z\delta_zH_{n,m,r}(x,y,z)\\
+f_2\mathbf{S}_y\delta_yH_{n,m,r}(x,y,z)
+f_3\mathbf{S}_x\delta_xH_{n,m,r}(x,y,z)
+f_4\mathbf{S}_y\delta_y\mathbf{S}_z\delta_zH_{n,m,r}(x,y,z)+
f_5\mathbf{S}_x\delta_x\mathbf{S}_z\delta_zH_{n,m,r}(x,y,z) \\
+f_6\mathbf{S}_x\delta_x\mathbf{S}_y\delta_yH_{n,m,r}(x,y,z)
+f_7\delta^2_zH_{n,m,r}(x,y,z)+f_8\delta^2_yH_{n,m,r}(x,y,z)+f_9\delta^2_xH_{n,m,r}(x,y,z)\\
+f_{10}\mathbf{S}_x\delta_x\mathbf{S}_y\delta_y\mathbf{S}_z\delta_zH_{n,m,r}(x,y,z)
+f_{11}\mathbf{S}_y\delta_y\delta_z^2H_{n,m,r}(x,y,z)
+f_{12}\mathbf{S}_y\delta_y\delta_x^2H_{n,m,r}(x,y,z)\\
+f_{13}\mathbf{S}_z\delta_z\delta_y^2H_{n,m,r}(x,y,z)+f_{14}\mathbf{S}_z\delta_z\delta_x^2H_{n,m,r}(x,y,z)+
f_{15}\mathbf{S}_x\delta_x\delta_y^2H_{n,m,r}(x,y,z) \\
+f_{16}\mathbf{S}_x\delta_x\delta_z^2H_{n,m,r}(x,y,z)
+f_{17}\mathbf{S}_x\delta_x\mathbf{S}_y\delta_y\delta_z^2H_{n,m,r}(x,y,z)
+f_{18}\mathbf{S}_x\delta_x\mathbf{S}_z\delta_z\delta_y^2H_{n,m,r}(x,y,z) \\
+f_{19}\mathbf{S}_y\delta_y\mathbf{S}_z\delta_z\delta_x^2H_{n,m,r}(x,y,z)+f_{20}\delta^2_y\delta^2_zH_{n,m,r}(x,y,z)
+f_{21}\delta^2_x\delta^2_yH_{n,m,r}(x,y,z)\\
+f_{22}\delta^2_x\delta^2_zH_{n,m,r}(x,y,z)+
f_{23}\mathbf{S}_x\delta_x\delta^2_y\delta^2_zH_{n,m,r}(x,y,z)+f_{24}\mathbf{S}_y\delta_y\delta^2_x\delta^2_zH_{n,m,r}(x,y,z)\\
+f_{25}\mathbf{S}_z\delta_z\delta^2_x\delta^2_yH_{n,m,r}(x,y,z)+
f_{26}\delta^2_x\delta^2_y\delta^2_zH_{n,m,r}(x,y,z)=0,
\end{multline*}
where
\begin{align*}
f_1&=\left( -a_{{1}}-2\,{\it e_2}-2\,{\it e_3}-b_{{1}}-b_{{4}}-a_{{4}}
 \right) z-i \left( -b_{{4}}a_{{1}}+b_{{1}}a_{{4}}+{\it e_3}\,a_{{4}}-b
_{{4}}{\it e_3}+{\it e_2}\,a_{{4}}-b_{{4}}{\it e_2} \right)
,\\
f_2&=\left( -a_{{1}}-2\,{\it e_2}-2\,{\it e_3}-b_{{1}}-b_{{4}}-a_{{4}}
 \right) y-i \left( -a_{{1}}{\it e_3}+b_{{1}}{\it e_3}-b_{{4}}{\it e_2}-b
_{{4}}a_{{1}}+b_{{1}}a_{{4}}+{\it e_2}\,a_{{4}} \right)
,\\
f_3&=i \left( -b_{{1}}{\it e_2}+b_{{4}}a_{{1}}-b_{{1}}a_{{4}}+a_{{1}}{\it e_3
}-b_{{1}}{\it e_3}+a_{{1}}{\it e_2} \right) + \left( -a_{{1}}-2\,{\it e_2
}-2\,{\it e_3}-b_{{1}}-b_{{4}}-a_{{4}} \right) x
,\\
f_4&=\left( -2\,z-i \left( -b_{{4}}+a_{{4}} \right)  \right) y+i \left( a_
{{1}}-b_{{1}} \right) z+b_{{1}}a_{{4}}+{\it e_2}\,a_{{4}}+b_{{4}}a_{{1}
}+b_{{4}}{\it e_2}
, \\
f_5&=\left( -2\,z-i \left( -b_{{4}}+a_{{4}} \right)  \right) x+i \left( a_
{{1}}-b_{{1}} \right) z+b_{{1}}a_{{4}}+b_{{4}}a_{{1}}
, \\
 f_6&=\left( -2\,y-i \left( -b_{{4}}+a_{{4}} \right)  \right) x+i \left( a_
{{1}}-b_{{1}} \right) y+a_{{1}}{\it e_3}+b_{{1}}{\it e_3}+b_{{4}}a_{{1}}
+b_{{1}}a_{{4}},  \\
f_7&=1/2\,b_{{4}}a_{{1}}+1/2\,b_{{1}}a_{{4}}+1/2\,{\it e_2}\,a_{{4}}+1/2\,{\it e_3}\,a_{{4}}+1/2\,b_{{4}}{\it e_2}+1/2\,b_{{4}}{\it e_3}\\
&+1/2\,i \left( a_{{1}}-b_{{1}}-a_{{4}}+b_{{4}} \right) z-{z}^{2},\\
f_8&=1/2\,b_{{4}}a_{{1}}+1/2\,a_{{1}}{\it e_3}+1/2\,b_{{1}}{\it e_3}+1/2\,b_{{1}}a_{{4}}+{\it e_2}\,{\it e_3}\\
&+1/2\,{\it e_2}\,a_{{4}}+1/2\,b_{{4}}{\it e_2}+1/2\,i \left( a_{{1}}-b_{{1}}-a_{{4}}+b_{{4}} \right) y-{y}^{2} ,\\
f_9&=1/2\,a_{{1}}{\it e_2}+1/2\,b_{{4}}a_{{1}}+1/2\,a_{{1}}{\it e_3}+1/2\,b_{{1}}{\it e_2}+1/2\,b_{{1}}{\it e_3}\\
&+1/2\,b_{{1}}a_{{4}}+1/2\,i \left( a_{{1}}-b_{{1}}-a_{{4}}+b_{{4}} \right) x-{x}^{2},  \\
f_{10}&=\left( a_{{4}}+b_{{4}} \right) x+ \left( a_{{1}}+b_{{1}} \right) z-i\left( b_{{4}}a_{{1}}-b_{{1}}a_{{4}} \right), \\
 f_{11}&=\left( 1/2\,b_{{4}}+1/2\,a_{{4}} \right) y+ \left( 1/2\,a_{{1}}+{\it e_2}+1/2\,b_{{1}} \right) z+1/2\,i \left( b_{{1}}a_{{4}}+{\it e_2}\,a_{{4}}-b_{{4}}a_{{1}}-b_{{4}}{\it e_2} \right) ,\\
f_{12}&=\left( 1/2\,a_{{1}}+1/2\,b_{{1}} \right) y+ \left( 1/2\,a_{{4}}+1/2\,b_{{4}}+{\it e_3} \right) x-1/2\,i \left( -b_{{1}}{\it e_3}+b_{{4}}a_{{1}}-b_{{1}}a_{{4}}+a_{{1}}{\it e_3} \right) , \\
f_{13}&=\left( 1/2\,b_{{4}}+1/2\,a_{{4}} \right) y+ \left( 1/2\,a_{{1}}+{\it e_2}+1/2\,b_{{1}} \right) z+1/2\,i \left( b_{{1}}a_{{4}}+{\it e_2}\,a_{{4}}-b_{{4}}a_{{1}}-b_{{4}}{\it e_2} \right) ,\\
f_{14}&=\left( 1/2\,a_{{1}}+1/2\,b_{{1}} \right) z+ \left( 1/2\,b_{{4}}+1/2\, a_{{4}} \right) x-1/2\,i \left( b_{{4}}a_{{1}}-b_{{1}}a_{{4}} \right)
, \\
f_{15}&=\left( 1/2\,a_{{1}}+1/2\,b_{{1}} \right) y+ \left( 1/2\,a_{{4}}+1/2\,b_{{4}}+{\it e_3} \right) x-1/2\,i \left( -b_{{1}}{\it e_3}+b_{{4}}a_{{1}}-b_{{1}}a_{{4}}+a_{{1}}{\it e_3} \right) , \\
f_{16}&=\left( 1/2\,a_{{1}}+1/2\,b_{{1}} \right) z+ \left( 1/2\,b_{{4}}+1/2\,a_{{4}} \right) x-1/2\,i \left( b_{{4}}a_{{1}}-b_{{1}}a_{{4}} \right)
,  \\
f_{17}&=\left( 1/2\,i \left( -b_{{4}}+a_{{4}} \right) +z \right) x-1/2\,b_{{1}}a_{{4}}-1/2\,i \left( a_{{1}}-b_{{1}} \right) z-1/2\,b_{{4}}a_{{1}}
,\\
f_{18}&= \left( 1/2\,i \left( -b_{{4}}+a_{{4}} \right) +z \right) x-1/2\,b_{{1}}a_{{4}}-1/2\,i \left( a_{{1}}-b_{{1}} \right) z-1/2\,b_{{4}}a_{{1}}
,
\end{align*}
\begin{align*}
f_{19}&=\left( 1/2\,i \left( -b_{{4}}+a_{{4}} \right) +z \right) x-1/2\,b_{{1}}a_{{4}}-1/2\,i \left( a_{{1}}-b_{{1}} \right) z-1/2\,b_{{4}}a_{{1}}, \\
f_{20}&=\left( 1/2\,z+1/4\,i \left( -b_{{4}}+a_{{4}} \right)  \right) y-1/4\,{\it e_2}\,a_{{4}}-1/4\,b_{{1}}a_{{4}} \\
&-1/4\,i \left( a_{{1}}-b_{{1}}  \right) z-1/4\,b_{{4}}a_{{1}}-1/4\,b_{{4}}{\it e_2},  \\
f_{21}&=\left( 1/2\,y+1/4\,i \left( -b_{{4}}+a_{{4}} \right)  \right) x-1/4\,a_{{1}}{\it e_3}-1/4\,i \left( a_{{1}}-b_{{1}} \right) y\\
&-1/4\,b_{{1}}a_{{4}}-1/4\,b_{{4}}a_{{1}}-1/4\,b_{{1}}{\it e_3}, \\
f_{22}&=\left( 1/2\,z+1/4\,i \left( -b_{{4}}+a_{{4}} \right)  \right) x-1/4\,b_{{1}}a_{{4}}-1/4\,i \left( a_{{1}}-b_{{1}} \right) z-1/4\,b_{{4}}a_{
{1}},\\
f_{23}&=\left( -1/4\,a_{{4}}-1/4\,b_{{4}} \right) x+ \left( -1/4\,a_{{1}}-1/4 \,b_{{1}} \right) z+1/4\,i \left( b_{{4}}a_{{1}}-b_{{1}}a_{{4}}
 \right) ,\\
 f_{24}&=\left( -1/4\,a_{{4}}-1/4\,b_{{4}} \right) x+ \left( -1/4\,a_{{1}}-1/4\,b_{{1}} \right) z+1/4\,i \left( b_{{4}}a_{{1}}-b_{{1}}a_{{4}}
 \right),\\
 f_{25}&=\left( -1/4\,a_{{4}}-1/4\,b_{{4}} \right) x+ \left( -1/4\,a_{{1}}-1/4\,b_{{1}} \right) z+1/4\,i \left( b_{{4}}a_{{1}}-b_{{1}}a_{{4}}
 \right),\\
 f_{26}&= \left( -1/8\,i \left( -b_{{4}}+a_{{4}} \right) -1/4\,z \right) x+1/8 \,b_{{4}}a_{{1}}+1/8\,b_{{1}}a_{{4}}+1/8\,i \left( a_{{1}}-b_{{1}}
 \right) z.
\end{align*}
\end{proposition}
\begin{proof}
Starting from $i=1$ to $i=6$, the coefficient of $E_{(l_1,l_2,l_3)}$ with $l_1+l_2+l_3=i$ and $l_1,l_2,l_3\in\{0,1,2\}$ is obtained by replacing  $H_{l_1,l_2,l_3}(x,y,z)$ in the divided-difference equation.
\end{proof}

\subsubsection{Coefficients of the three-term recurrence relations satisfied by bivariate continuous Hahn polynomials and new family of monic bivariate continuous Hahn polynomials}

The column vectors of both families of bivariate continuous Hahn polynomials satisfy three-term recurrence relations of the form \eqref{RRTT}. As in the previous cases, the coefficients of the matrices can be obtained by using the same procedure as described for bivariate Racah polynomials, by using \eqref{eq:matricesgp} with $\lambda_{n}=n(a_{1}+2e_{2}+b_{3}+a_{3}+b_{1}+n-1)$. As a consequence, we can introduce new families of bivariate continuous dual Hahn polynomials by choosing appropriately the matrix $G_{n,n}$, where the bases $\textbf{x}^n={\mathbf{F}}_{n}=(x^{n-k}y^{k})_{k=0,\dots,n}$.

\begin{proposition}
The matrices ${\mathbf{S}}_{n}$ of size $(n+1)\times n$ and $\mathbf{T}_n$ of size $(n+1)\times (n-1)$ have the same structure as in \eqref{eq:sn} and \eqref{eq:tn} respectively, and their coefficients are given in terms of of the polynomial coefficients of the equation \eqref{eqdiv-diff-biv-CDH} as
\begin{align*}
s_{k,k}&=\frac{1}{2} i (n-k+1) \left(2 \left(a_1 \left(b_3+e_2\right)-b_1
   \left(a_3+e_2\right)\right) \right. \\
   &\left. +\left(a_1-a_3-b_1+b_3\right) (n-k)+2 (k-1) \left(a_1-b_1\right)\right), \quad k=1,\dots,n, \\
s_{k+1,k}&=\frac{1}{2} i k \left(a_3 \left(-2 b_1-2 e_2+k-2 n+1\right)+a_1 \left(2 b_3+k-1\right)+2 b_3 e_2 \right. \\
& \left. -b_1 k -b_3 k+2 b_3 n+b_1-b_3\right), \quad k=1,\dots,n,\\
t_{k,k}&=\frac{1}{12} (n-k) (n-k+1) \left(-2 (k-n+1) \left(a_1+a_3+b_1+b_3+2 e_2\right)+6 \left(b_1
   \left(a_3+e_2\right) \right. \right. \\
   & \left. \left. +a_1 \left(b_3+e_2\right)\right)+6 (k-1) \left(a_1+b_1\right)-(k-n+1) (3
   k+n-6)\right), \quad k=1,\dots,n-1, \\
t_{k+1,k}&=\frac{1}{2} k (n-k) \left(\left(a_3+b_3\right) (n-k-1)+(k-1) \left(a_1+b_1\right)+2 \left(a_3 b_1+a_1
   b_3\right) \right. \\
   & \left. +(k-1) (n-k-1)\right), \quad k=1,\dots,n-1,  \\
t_{k+2,k}&=\frac{1}{12} k (k+1) \left(2 (k-1) \left(a_1+a_3+b_1+b_3+2 e_2\right)+6 \left(b_3 \left(a_1+e_2\right)+a_3
   \left(b_1+e_2\right)\right) \right. \\
   & \left. -6 \left(a_3+b_3\right) (k-n+1)+(k-1) (4 n-3 k+-6)\right), \quad k=1,\dots,n-1.
\end{align*}
\end{proposition}
Moreover, the matrices of leading coefficients of continuous Hahn polynomials defined in \eqref{eq:bchp1} $H_{n,m}(a_1,e_2,a_3;b_1,b_3;x,y)$ are given by
\begin{equation}\label{eq:gnnbr5}
G_{n,n}(a_{1},e_{2},a_{3};b_{1},b_{3})=
\begin{pmatrix}
g_{r,s}(n,a_{1},e_{2},a_{3};b_{1},b_{3})
\end{pmatrix}_{0 \leq r,s \leq n},
\end{equation}
where
\begin{align*}
g_{r,s}(n,a_{1},e_{2},a_{3};b_{1},b_{3})=\begin{cases}
\displaystyle{0}, & r<s, \\[1mm]
\displaystyle{(-1)^{r-s} \binom{n-r}{s-r}  (a-{1}+b_{1}+2 e_{2}-r+n-1)_{n-s}} \\
\quad \times \displaystyle{(a_{1}+b_{1}-s+n)_{s-r}\,(a_{1}+a_{3}+b_{1}+b_{3}+2 e_{2}-r+2 n-1)_{r}}, & r \geq s,
\end{cases}
\end{align*}
and the matrices of leading coefficients of bivariate continuous Hahn polynomials defined in \eqref{eq:bchp1} $\bar{H}_{n,m}(a_1,e_2,a_3;b_1,b_3;x,y)$ are given by
\begin{equation}\label{eq:gnnbr6}
\bar{G}_{n,n}(a_{1},e_{2},a_{3};b_{1},b_{3})=
\begin{pmatrix}
\bar{g}_{r,s}(n,a_{1},e_{2},a_{3};b_{1},b_{3})
\end{pmatrix}_{0 \leq r,s \leq n},
\end{equation}
where
\begin{align*}
\bar{g}_{r,s}(n,a_{1},e_{2},a_{3};b_{1},b_{3})=\begin{cases}
\displaystyle{0}, & r>s, \\[1mm]
\displaystyle{(-1)^{r-s} \binom{r}{s} (a_{3}+b_{3}+s)_{r-s} (a_{3}+b_{3}+2 e_{2}+i-1)_{s} } \\
\quad \times \displaystyle{ (a_{1}+a_{3}+b_{1}+b_{3}+2 e_{2}+r+n-1)_{n-r}}, & r \leq s.
\end{cases}
\end{align*}
As a consequence, both families of bivariate continuous dual Hahn polynomials can be generated from the three-term recurrence relations they satisfy by using Theorem \ref{eq:matricesTTRR}, where for these specific families we have $x_{1}(x)=x$ and $y_{2}(y)=y$. Finally, if we consider $G_{n,n}$ as the identity matrix we can introduce the family of monic bivariate continuous dual Hahn polynomials by using Corollary \ref{eq:matricesTTRRmonic}.

\section{Acknowledgments}
The first author is indebted to the AIMS-Cameroon 2014--2015 and 2015--2016  tutor fellowships. The third author acknowledges support from the AIMS-Cameroon 2014--2015 research grant and the hospitality and financial support during his visit to Universidade de Vigo in July 2015. This work has been partially supported for the fourth and fifth authors by the Ministerio de Econom\'{\i}a y Competi\-tividad of Spain under grant MTM2012--38794--C02--01, co-financed by the European Community fund FEDER. The last author thanks the hospitality of the African Institute for Mathematical Sciences (AIMS-Cameroon), where a significant part of this research was performed during his visits in November 2014, and May and June 2015.

\end{document}